\documentclass[12pt]{article}
\usepackage{mathrsfs}
\usepackage{graphicx}
\usepackage{subfig}
\usepackage{epstopdf}
\usepackage{amsmath}
\usepackage{amsfonts}
\usepackage{amssymb, amsmath, cite}
\usepackage{cases}

\newtheorem{theorem}{Theorem}
\newtheorem{lemma}{Lemma}
\newtheorem{proposition}{Proposition}
\newtheorem{remark}{Remark}

\newtheorem{assumption}{Assumption}

\newbox\qedbox
\newenvironment{proof}{\smallskip\noindent{\bf Proof.}\hskip \labelsep}%
                        {\hfill\penalty10000\copy\qedbox\par\medskip}

\newcommand{\bfR}{{\Bbb R}}
\newcommand{\bfC}{{\Bbb C}}

\newcommand{\ii}{\text{i}}
\newcommand{\e}{\text{e}}
\newcommand{\dd}{\text{d}}

\newcommand{\nn}{\nonumber}
\newcommand\be{\begin{equation}}
\newcommand\ee{\end{equation}}
\newcommand{\bea}{\begin{eqnarray}}
\newcommand{\eea}{\end{eqnarray}}
\newcommand\berr{\begin{eqnarray*}}
\newcommand\eerr{\end{eqnarray*}}
{\setlength\arraycolsep{2pt}

\begin{document}

\title{Initial-boundary value problem and long-time asymptotics for the Kundu--Eckhaus equation on the half-line}
\author{ Boling Guo$^{a}$,\, Nan Liu$^{b,}$\footnote{Corresponding author.}\\
$^a${\small{\em Institute of Applied Physics and Computational Mathematics,  Beijing 100088, P.R. China}} \\
$^b${\small{\em The Graduate School of China Academy of Engineering Physics, Beijing 100088, P.R. China}}\setcounter{footnote}{-1}\footnote{E-mail addresses: gbl@iapcm.ac.cn (B. Guo), ln10475@163.com (N. Liu).}
}

\date{}
\maketitle

\begin{quote}
{{{\bfseries Abstract.} The initial-boundary value problem for the Kundu--Eckhaus equation on the half-line is considered in this paper by using the Fokas method. We will show that the solution $u(x,t)$ can be expressed in terms of the solution of a matrix Riemann--Hilbert problem formulated in the complex $k$-plane. Furthermore, based on a nonlinear steepest descent analysis of the associated Riemann--Hilbert problem, we can give the precise asymptotic formulas for the solution of the  Kundu--Eckhaus equation on the half-line.

}

 {\bf Keywords:} Kundu--Eckhaus equation, Riemann--Hilbert problem,  Initial-boundary value problem, Nonlinear steepest descent method, Long-time asymptotics.}
\end{quote}

\section{Introduction}
\setcounter{equation}{0}
The discovery of the Lax pairs \cite{PDL} for nonlinear evolutionary equations and the inverse scattering transform (IST) method for solving initial-value problems on the whole line, which is one of the most important developments in the study of nonlinear systems, pioneered at 1967 by Kruskal et al \cite{CSG} in researching the Korteweg-de Vries equation, turn out to be very successful. After decades of development, Fokas \cite{F1,F2} has developed a new powerful approach in 1997 which generalizes the IST formalism from initial-value to initial-boundary value (IBV) problems such that the soliton theory achieves an important development. This approach combines the main insights of the IST with elements of the theory of Riemann--Hilbert problems, and uses essentially complex analytic properties to eliminate unknown boundary values from the solution representation. Since then there have been a wide range of explorations on the analysis of boundary value problems for several of the most important integrable equations with $2\times2$ Lax pairs, such as the modified Korteweg-de Vries equation \cite{AB}, the nonlinear Schr\"oinger (NLS) equation \cite{F3}, the derivative Schr\"odinger equation \cite{JLe}.

In this paper, we use the general method for solving the IBV problem announced in \cite{F1,F2} to study the Kundu--Eckhaus (KE) equation \cite{AK}, which contains quintic nonlinearity and the Raman effect in nonlinear optics
\begin{equation}\label{1.1}
\ii u_t+u_{xx}-2|u|^2u+4\beta^2|u|^4u+4\ii\beta(|u|^2)_xu=0,~~\beta\in\bfR\setminus\{0\}.
\end{equation}
on the half-line, where $u(x,t)$ is the complex smooth envelop function. $\beta$ is a real constant, $\beta^2$ is the quintic nonlinear coefficient, and the last term represents the Raman effect, which is responsible for the self-frequency shift.
For equation \eqref{1.1}, a series of important results have been obtained, such as the gauge connections between
equation \eqref{1.1} and other soliton equations \cite{AK}, the Lax pair and the Hamiltonian structure \cite{GXG}, the higher-order rogue wave solutions\cite{WX}, and soliton solutions through the Darboux transformation \cite{DQ}. It is noted that the Riemann--Hilbert (RH) problem and long-time asymptotics for equation \eqref{1.1} with decaying initial value on the line was studied in \cite{ZQZ}, which motivates our present analysis.

Our purpose here is to consider the IBV problem for the KE equation \eqref{1.1} via the Fokas method on the half-line, that is, in the domain
$$\Omega=\{0<x<\infty,~0<t<T\}$$
and $T\leq\infty$ is a given positive constant. We will denote the initial data,  Dirichlet and Neumann boundary values of \eqref{1.1} as follows:
\bea
u(x,0)=u_0(x),~u(0,t)=g_0(t),~u_x(0,t)=g_1(t).
\eea
Assuming that the solution $u(x,t)$ exists, we show that it can be represented in terms of the solution of a matrix RH problem formulated in the plane of the complex spectral parameter $k$, with jump matrices given in terms of spectral functions $a(k)$, $b(k)$ (obtained from the initial data $u(x,0)=u_0(x)$) and $A(k)$, $B(k)$ (obtained from the boundary values $u(0,t)=g_0(t)$ and $u_x(0,t)=g_1(t)$).

An important advantage of this representation obtained is that it yields precise information about the long-time asymptotics of the solution by using the nonlinearization of the steepest descent method. The nonlinear steepest descent method was first introduced by Deift and Zhou in \cite{PD}, which was inspired by earlier works of Manakov and Its. The method has since then proved successful in determining asymptotic formulas for a large range of other initial-value problems for various integrable equations \cite{AB1,AB2,AB3,PD1,RB,XJ,HL,ZQZ}. By combining the ideas of \cite{PD} with the unified transform formalism of \cite{F1}, it is also possible to study asymptotics of solutions of IBV problems for nonlinear integrable PDEs \cite{JL1,JL2}.

Developing and extending the methods using in \cite{JL1,JL2}, we will further derive the long-time asymptotics of the solution $u(x,t)$ on the half-line by performing a nonlinear steepest descent analysis of the associated RH problem. Compared with the analysis of the initial problem considered in \cite{ZQZ}, the half-line problem is more involved. For example, the relevant RH problem for the Cauchy problem \eqref{1.1} only has a jump across $\bfR$, whereas the RH problem for the IBV problem also has jump across $\ii\bfR$, and the jump across this line involves the spectral function $\Gamma(k)$. Moreover, during the asymptotic analysis, one should find an analytic approximation $\Gamma_a(t,k)$ of $\Gamma(k)$ and deform the contour to eliminate the part of the jump that involves $\Gamma_a(t,k)$.

The outline of the paper is following. In Section 2, we study the direct spectral problem formulated in terms of the simultaneous spectral analysis of the associated Lax pair: we define appropriate eigenfunctions and spectral functions and study their properties.  We formulate the main RH problem in Section 3 and show that $u(x,t)$ can be expressed in terms of the solution of this $2\times2$ matrix RH problem. In section 4, we present the detailed derivation of the long-time asymptotics for the solution of KE equation.

\section{Spectral theory}
\setcounter{equation}{0}
\setcounter{lemma}{0}
\setcounter{theorem}{0}

\subsection{The Lax pair}
Let
\be\label{2.1}
\sigma_3={\left( \begin{array}{cc}
1 & 0 \\[4pt]
0 & -1\\
\end{array}
\right )},\quad
Q(x,t)={\left( \begin{array}{cc}
0 & u(x,t) \\[4pt]
\bar{u}(x,t) & 0 \\
\end{array}
\right )}.
\ee
The KE equation \eqref{1.1} is the condition of compatibility of \cite{WX}
\begin{equation}\label{2.2}
\begin{aligned}
&\phi_x+\ii k\sigma_3\phi=Q_1\phi,\\
&\phi_t+2\ii k^2\sigma_3\phi=Q_2\phi,\\
\end{aligned}
\end{equation}
where $\phi(x,t;k)$ is a $2\times2$ matrix-valued function and $k\in\bfC$ is the spectral parameter,
\begin{eqnarray}
Q_1&=&Q-\ii\beta Q^2\sigma_3={\left( \begin{array}{cc}
-\ii\beta|u|^2 & u \\[4pt]
\bar{u} & \ii\beta|u|^2\\
\end{array}
\right )},\label{2.3}\\
Q_2&=&4\ii\beta^2Q^4\sigma_3-2\beta Q^3-\ii Q^2\sigma_3+2kQ-\ii Q_x\sigma_3+\beta(Q_xQ-QQ_x)\\
&=&{\left( \begin{array}{cc}
4\ii\beta^2|u|^4-\ii|u|^2+\beta(u_x\bar{u}-u\bar{u}_x)  & 2ku-2\beta |u|^2u+\ii u_x\\[4pt]
2k\bar{u}-2\beta |u|^2\bar{u}-\ii\bar{u}_x & -4\ii\beta^2|u|^4+\ii|u|^2-\beta(u_x\bar{u}-u\bar{u}_x) \nn\\
\end{array}
\right )}.\label{2.4}
\end{eqnarray}
Starting with this Lax pair and following steps similar to the ones used in \cite{JLe}, we find that in order to have a
function satisfying, within its region of boundedness,
\be
\mu=I+O\bigg(\frac{1}{k}\bigg),\quad k\rightarrow\infty,\label{2.5}
\ee
we introduce a new eigenfunction $\mu(x,t;k)$ by
\begin{equation}\label{2.6}
\phi(x,t;k)=\e^{\ii\int_{(0,0)}^{(x,t)}\Delta\sigma_3}\mu(x,t;k)\e^{\ii\int_{(0,0)}^{(\infty,0)}\Delta\sigma_3}\e^{-\ii(k x+2k^2t)\sigma_3},
\end{equation}
where $\Delta$ is the closed real-valued one-form
\be\label{2.7}
\Delta(x,t)=-\beta|u|^2\dd x+\bigg(4\beta^2|u|^4-\ii\beta(u_x\bar{u}-u\bar{u}_x)\bigg)\dd t.
\ee
Then we obtain the equivalent Lax pair
\begin{equation}\label{2.8}
\begin{aligned}
&\mu_x+\ii k[\sigma_3,\mu]=U_1\mu,\\
&\mu_t+2\ii k^2[\sigma_3,\mu]=U_2\mu,\\
\end{aligned}
\end{equation}
where the matrices $U_1$ and $U_2$ are defined respectively by
\bea
U_1&=&\begin{pmatrix}
0 & u\e^{-2\ii\int_{(0,0)}^{(x,t)}\Delta}\\[4pt]
\bar{u}\e^{2\ii\int_{(0,0)}^{(x,t)}\Delta} & 0
\end{pmatrix},\label{2.9}\\
U_2&=&\begin{pmatrix}
-\ii|u|^2 & \big(2ku+\ii u_x-2\beta|u|^2u\big)\e^{-2\ii\int_{(0,0)}^{(x,t)}\Delta}\\[4pt]
\big(2k\bar{u}-\ii\bar{u}_x-2\beta|u|^2\bar{u}\big)\e^{2\ii\int_{(0,0)}^{(x,t)}\Delta} & \ii|u|^2
\end{pmatrix}.\label{2.10}
\eea
Equations \eqref{2.8} can be written in differential form as
\be\label{2.11}
\dd\big(\e^{\ii(k x+2k^2t)\hat{\sigma}_3}\mu(x,t;k)\big)=w(x,t;k),
\ee
where
\be\label{2.12}
w(x,t;k)=\e^{\ii(k x+2k^2t)\hat{\sigma}_3}\big(U_1(x,t)\dd x+U_2(x,t;k)\dd t\big)\mu(x,t;k),
\ee
and $\hat{\sigma}_3$ denote the operator which act on a $2\times2$ matrix $X$ by $\hat{\sigma}_3X=[\sigma_3,X]$, then $\e^{\hat{\sigma}_3}X=\e^{\sigma_3} X\e^{-\sigma_3}$.

\subsection{Bounded and analytic eigenfunctions}
Let equation \eqref{1.1} be valid for the half-line domain $\Omega=\{0<x<\infty,0<t<T\}$. Assume that $u(x,t)$ is sufficiently smooth functions of $(x,t)$ in the half-line domain $\Omega$ which decay as $x\rightarrow\infty$.

\begin{figure}[htbp]
\begin{minipage}[t]{0.3\linewidth}
\centering
\includegraphics[width=2.15in]{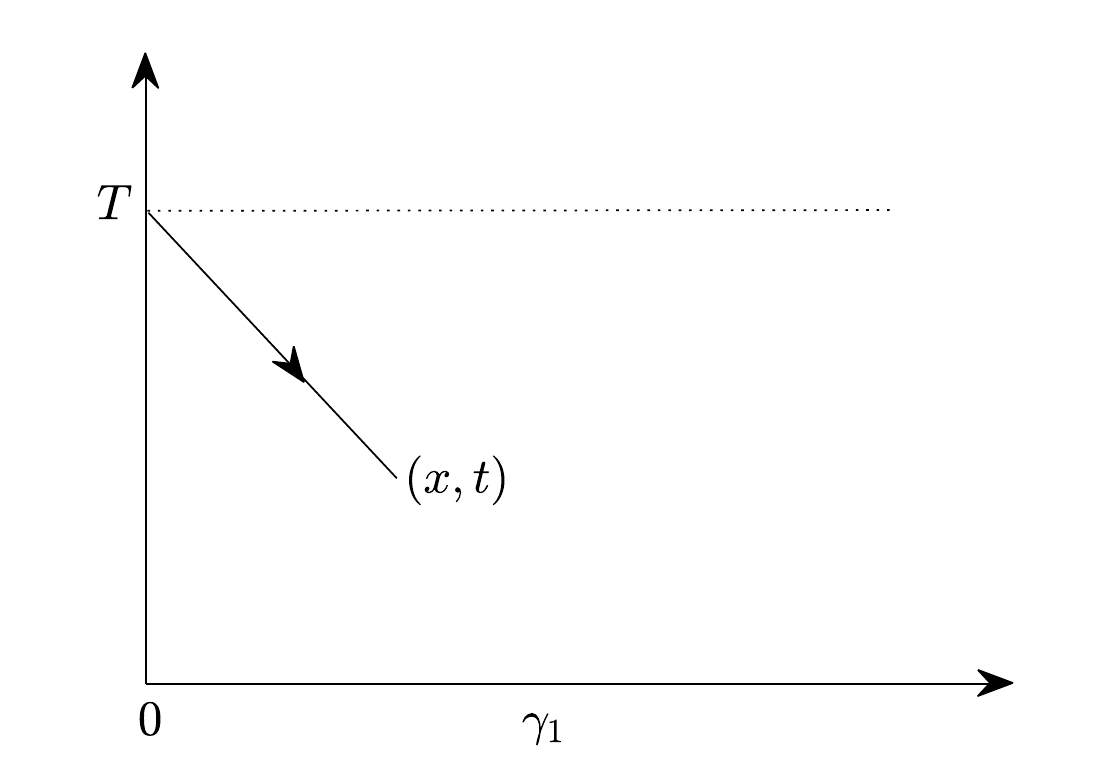}
\label{fig:side:a}
\end{minipage}%
\begin{minipage}[t]{0.3\linewidth}
\centering
\includegraphics[width=2.15in]{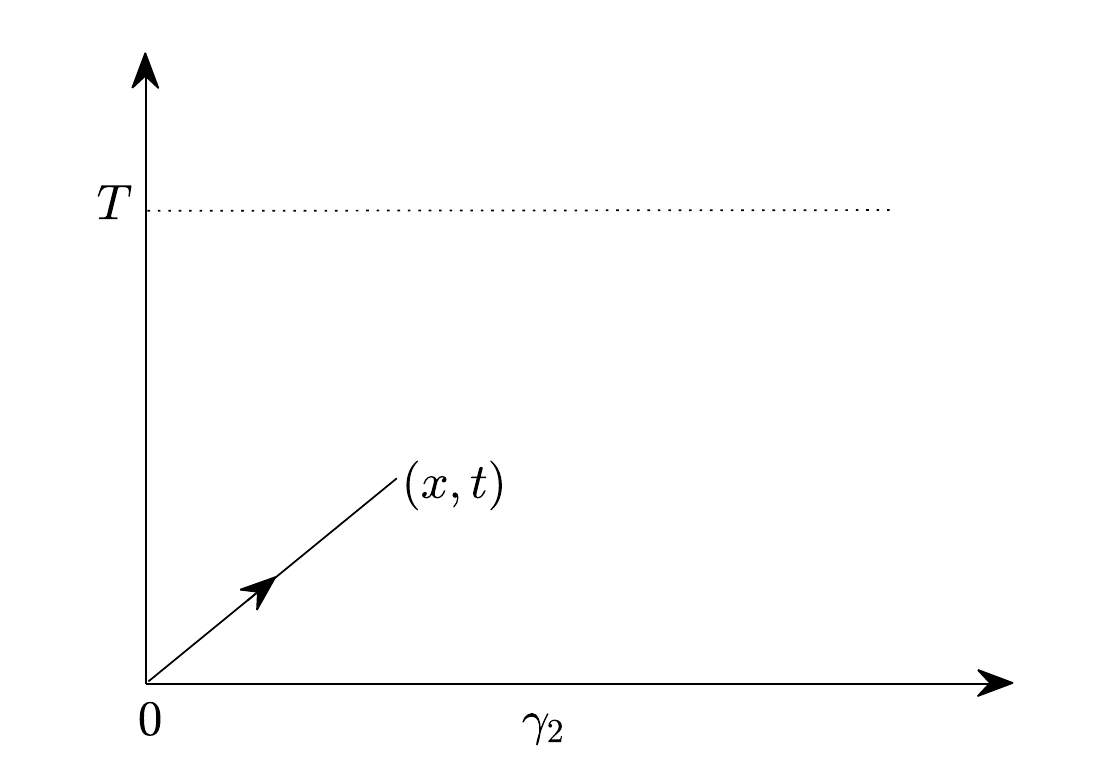}
\label{fig:side:b}
\end{minipage}
\begin{minipage}[t]{0.3\linewidth}
\centering
\includegraphics[width=2.15in]{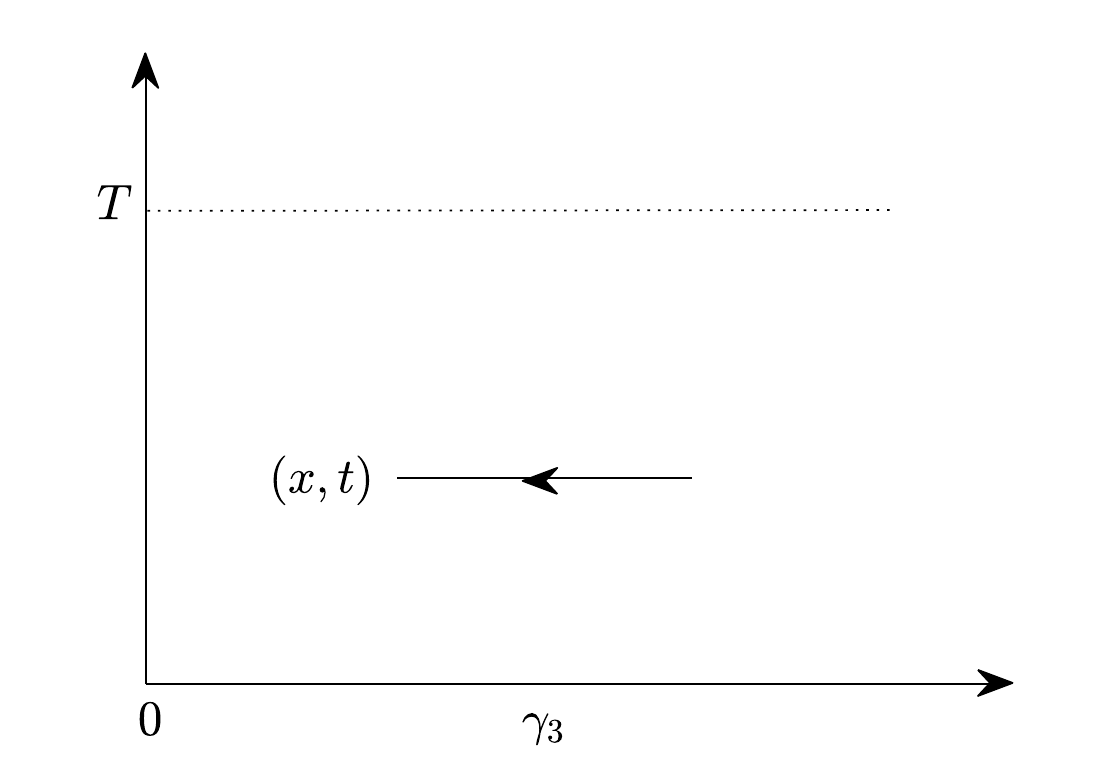}
\label{fig:side:c}
\end{minipage}
\caption{The contours $\gamma_1$, $\gamma_2$ and $\gamma_3$ in the $(x,t)$-plane.}\label{fig1}
\end{figure}

Three eigenfunctions $\{\mu_j\}_1^3$ of equation \eqref{2.11} are defined by the following Volterra integral equations
\be\label{2.13}
\mu_j(x,t;k)=I+\int_{\gamma_j}\e^{-\ii(k x+2k^2t)\hat{\sigma}_3}w_j(x',t';k),~j=1,2,3,
\ee
where $w_j$ is given by \eqref{2.12} with $\mu$ replaced with $\mu_j$, and the contours $\{\gamma_j\}^3_1$ denote the smooth curves from $(x_j,t_j)$ to $(x,t)$, and $(x_1,t_1)=(0,T)$, $(x_2,t_2)=(0,0)$, $(x_3,t_3)=(\infty,t)$. Since the one-form $w$ is exact, the integral on the right-hand side of equation \eqref{2.13} is independent of the path of
integration. We choose the particular contours shown in Fig. \ref{fig1}. This choice implies the following inequalities on the contours:
\begin{eqnarray}\label{2.14}
&&\gamma_1:~x'-x\leq0,\quad t'-t\geq0,\nn\\
&&\gamma_2:~x'-x\leq0,\quad t'-t\leq0,\\
&&\gamma_3:~x'-x\geq0.\nn
\end{eqnarray}

The second column of the matrix equation \eqref{2.13} involve the exponential $\e^{2\ii[k(x'-x)+2k^2(t'-t)]}$. Using the above inequalities it follows that this exponential is bounded in the following regions of
the complex $k$-plane:
\begin{eqnarray}\label{2.15}
&&\mu_1:\{\text{Im}k\leq0\cap\text{Im}k^2\geq0\},\nn\\
&&\mu_2:\{\text{Im}k\leq0\cap\text{Im}k^2\leq0\},\\
&&\mu_3:\{\text{Im}k\geq0\}.\nn
\end{eqnarray}
These boundedness properties imply that $[\mu_1]_2$, $[\mu_2]_2$, and $[\mu_3]_2$ are bounded and analytic for
$k\in D_3, D_4$ and $\bfC_+$, respectively. Similar conditions are valid for the first column vectors. Thus,
the functions $\{\mu_j\}^3_1$ are bounded and analytical for $k\in\bfC$ while $k$ belongs to
\begin{eqnarray}\label{2.16}
&&\mu_1:~(D_2,D_3),\nn\\
&&\mu_2:~(D_1,D_4),\\
&&\mu_3:~(\bfC_-,\bfC_+),\nn
\end{eqnarray}
where $D_n$ denotes $n$th quadrant, $1\leq n\leq4$, $\bfC_+$ and $\bfC_-$ denote the upper and lower half complex $k$-plane, respectively.

In fact, for $x=0$, $\mu_1(0,t;k)$ has enlarged the domain of boundedness: $(D_2\cup D_4,D_1\cup D_3)$, $\mu_2(0,t;k)$ have enlarged the domain of boundedness: $(D_1\cup D_3,D_2\cup D_4)$, and $\mu_2(x,0;k)$ have enlarged the domain of boundedness: $(D_1\cup D_2,D_3\cup D_4)$. We also note that the functions $\mu_1(x,t;k)$ ($T<\infty$) and $\mu_2(x,t;k)$ are entire functions of $k$. Moreover, in their corresponding regions of boundedness,
$$\mu_j(x,t;k)=I+O\bigg(\frac{1}{k}\bigg),\quad k\rightarrow\infty,\quad j=1,2,3.$$

\subsection{Spectral functions}
Any two solutions $\mu$ and $\tilde{\mu}$ of \eqref{2.11} are related by an equation of the form
\be\label{1}
\mu(x,t;k)=\tilde{\mu}(x,t;k)\e^{-\ii(k x+2k^2t)\hat{\sigma}_3}C_0(k),
\ee
where $C_0(k)$ is a $2\times2$ matrix independent of $x$ and $t$. Indeed, let $\phi$ and $\tilde{\phi}$ be the solutions of
equation \eqref{2.2} corresponding to $\mu$ and $\tilde{\mu}$ according to \eqref{2.6}. Then, since the first and second
columns of a solution of \eqref{2.2} satisfy the same equation, there exists a $2\times2$ matrix $C_1(k)$ independent of $x$ and $t$ such that
\berr
\phi(x,t;k)=\tilde{\phi}(x,t;k)C_1(k).
\eerr
It follows that \eqref{1} is satisfies with $C_0(k)=\e^{-\ii\int_{(0,0)}^{(\infty,0)}\Delta\hat{\sigma}_3}C_1(k)$.

We define $s(k)$ and $S(k)$ by the relations
\bea
&&\mu_3(x,t;k)=\mu_2(x,t;k)\e^{-\ii(k x+2k^2t)\hat{\sigma}_3}s(k),\label{2.17}\\
&&\mu_1(x,t;k)=\mu_2(x,t;k)\e^{-\ii(k x+2k^2t)\hat{\sigma}_3}S(k).\label{2.18}
\eea
Evaluation of \eqref{2.17} and \eqref{2.18} at $(x,t)=(0,0)$ and $(x,t)=(0,T)$ gives the following
expressions
\be\label{2.19}
s(k)=\mu_3(0,0;k),~S(k)=\mu_1(0,0;k)=
\bigg(\e^{2\ii k^2T\hat{\sigma}_3}\mu_2(0,T;k)\bigg)^{-1}.
\ee
Hence, the functions $s(k)$ and $S(k)$ can be obtained respectively from the evaluations at $x=0$ and at $t=T$ of the functions $\mu_3(x,0;k)$ and $\mu_2(0,t;k)$, which satisfy the linear integral equations
\bea
\mu_3(x,0;k)&=&I-\int_x^\infty\e^{\ii k(x'-x)\hat{\sigma_3}}(U_1\mu_3)(x',0;k)\dd x',\label{2.20}\\
\mu_2(0,t;k)&=&I+\int_0^t\e^{2\ii k^2(t'-t)\hat{\sigma_3}}(U_2\mu_2)(0,t';k)\dd t'.\label{2.21}
\eea
By evaluating \eqref{2.9} and \eqref{2.10} at $t=0$ and $x=0$, respectively, we find that
\bea
U_1(x,0;k)&=&\begin{pmatrix}
0 & u_0\e^{2\beta\ii\int_0^x|u_0|^2\dd x'}\\[4pt]
\bar{u}_0\e^{-2\beta\ii\int_0^x|u_0|^2\dd x'} & 0
\end{pmatrix},\label{2.22}\\
U_2(0,t;k)&=&\begin{pmatrix}
-\ii|g_0|^2 & \big(2kg_0+\ii g_1-2\beta|g_0|^2g_0\big)\e^{-2\ii\int_{0}^{t}\Delta_2(0,t')\dd t'}\\[4pt]
\big(2k\bar{g}_0-\ii\bar{g}_1-2\beta|g_0|^2\bar{g}_0\big)\e^{2\ii\int_{0}^{t}\Delta_2(0,t')\dd t'} & \ii|g_0|^2
\end{pmatrix},\nn\\
&&\Delta_2(0,t)=4\beta^2|g_0|^4-\ii\beta(g_1\bar{g}_0-g_0\bar{g}_1).\label{2.23}
\eea
Furthermore, the fact that $U_1$ and $U_2$ are traceless together with $\mu_j=I+O(\frac{1}{k})$ imply that $\det\mu_j=1$. Thus, we have
\be\label{2.24}
\det s(k)=\det S(k)=1.
\ee
\subsection{The symmetry properties}
\begin{proposition}\label{pro2.1}
For $j=1,2,3$, the function $\mu_j(x,t;k)$ satisfies the following symmetry relations
\be\label{2.25}
\mu_{11}(x,t;k)=\overline{\mu_{22}(x,t;\bar{k})},~~
\mu_{21}(x,t;k)=\overline{\mu_{12}(x,t;\bar{k})}.
\ee
\end{proposition}
\begin{proof}
For~a~$2\times2$~matrix~A, we define a operator $T$ as follows:
$$TA=\begin{pmatrix}
\bar{a}_{22} & \bar{a}_{21}\\[4pt]
\bar{a}_{12} & \bar{a}_{11}\\
\end{pmatrix},~\mbox{where}~A=\begin{pmatrix}
a_{11} & a_{12}\\[4pt]
a_{21} & a_{22}\\
\end{pmatrix}.$$
Equation \eqref{2.25} is a consequence of the symmetries $(TU_j)(\bar{k})=U_j(k)$ valid for $j=1,2$.
\end{proof}

Particularly, it follows from \eqref{2.25} that the spectral matrices $s(k)$ and $S(k)$ can be written as
\be\label{2.26}
s(k)=\begin{pmatrix}
\overline{a(\bar{k})} & b(k)\\[4pt]
\overline{b(\bar{k})} & a(k)\\
\end{pmatrix},\quad
S(k)=\begin{pmatrix}
\overline{A(\bar{k})} & B(k)\\[4pt]
\overline{B(\bar{k})} & A(k)\\
\end{pmatrix}.
\ee

\subsection{The global relation}
The initial and boundary values of a solution of the KE equation are not independent. However, it turns out that the spectral functions $s(k)$ and $S(k)$ can be determined by the initial datum and boundary values, respectively. Thus, the spectral functions must satisfy a surprisingly simple relation. Indeed, it follows from \eqref{2.17} and \eqref{2.18} that
\be\label{2.27}
\mu_1(x,t;k)\e^{-\ii(k x+2k^2t)\hat{\sigma}_3}(S^{-1}(k)s(k))=\mu_3(x,t;k).
\ee
Since $\mu_1(0,T;k)=I$, evaluation at $(0,T)$ yields the following relation:
\be\label{2.28}
-I+S^{-1}(k)s(k)+\e^{2\ii k^2T\hat{\sigma}_3}\int_0^\infty
\e^{\ii k x'\hat{\sigma}_3}(U_1\mu_3)(x',T;k)\dd x'=0.
\ee
The (12) component of this equation is
\be\label{2.29}
B(k)a(k)-A(k)b(k)=\e^{4\ii k^2T}c^+(k),~\mbox{Im}k\geq0,
\ee
for $T<\infty$, where
$$c^+(k)=\int_0^\infty\e^{2\ii k x'}(U_1\mu_3)_{12}(x',T;k)\dd x'$$
is a function analytic for Im$k>0$ which is $O(\frac{1}{k})$ as $k\rightarrow\infty$.

In the case $T=\infty$, the global relation \eqref{2.29} reads
\be\label{2.30}
B(k)a(k)-A(k)b(k)=0,~k\in D_1.
\ee

\subsection{The properties of spectral functions}
Let $[A]_1$ ($[A]_2$) be denoted as the first (second) column of a $2\times2$ matrix $A$.

Given $u_0(x)\in \mathcal{S}([0,\infty))$, we define the map
\berr
\mathcal{J}:\{u_0(x)\}\mapsto\{a(k),b(k)\}
\eerr
by
\be\label{2.31}
\begin{pmatrix}
b(k)\\[4pt]
a(k)\\\end{pmatrix}=[\mu_3(0;k)]_2,\quad \mbox{Im}k\geq0,
\ee
where
$$\mu_3(x;k)=I-\int_x^\infty\e^{\ii k(x'-x)\hat{\sigma}_3}\big(U_1(x',0;k)\mu_3(x';k)\big)\dd x',$$
$U_1(x,0;k)$ is given in terms of $u_0(x)$ by \eqref{2.22}.
The analysis of the Volterra linear integral equation gives the following properties of $a(k)$ and $b(k)$.

(i) $a(k)$ and $b(k)$ are analytic for Im$k>0$, continuous, and bounded for Im$k\geq0$;

(ii) $a(k)=1+O(\frac{1}{k})$, $b(k)=O(\frac{-\frac{\ii}{2}u_0(0)}{k})$, as $k\rightarrow\infty$;

(iii) $a(k)\overline{a(\bar{k})}-b(k)\overline{b(\bar{k})}=1$;

(iv) The map $\mathcal{K}:\{a(k),b(k)\}\mapsto \{u_0(x)\}$, inverse to $\mathcal{J}$, is defined as follows:
\be\label{2.32}
u_0(x)=2\ii m(x)\e^{-8\beta\ii\int_0^x|m(x')|^2\dd x'},\quad m(x)=\lim_{k\rightarrow\infty}(k M^{(x)}(x;k))_{12},
\ee
where $M^{(x)}(x;k)$ is the unique solution of the matrix  RH problem:

$\bullet$
$
M^{(x)}(x;k)=\left\{
\begin{aligned}
&M^{(x)}_+(x;k),\quad k\in\bfC_+,\\
&M^{(x)}_-(x;k),\quad k\in\bfC_-,
\end{aligned}
\right.
$
is a sectionally meromorphic function;

$\bullet$  $M^{(x)}(x;k)$ satisfies the jump condition
\be\label{2.33}
M^{(x)}_+(x;k)=M^{(x)}_-(x;k)J^{(x)}(x,k),~~k\in\bfR,
\ee
where
\be\label{2.34}
J^{(x)}(x,k)=\begin{pmatrix}
\frac{1}{a(k)\overline{a(\bar{k})}} ~& \frac{b(k)}{\overline{a(\bar{k})}}\e^{-2\ii k x} \\[4pt]
-\frac{\overline{b(\bar{k})}}{a(k)}\e^{2\ii k x} ~& 1 \\
\end{pmatrix};
\ee

$\bullet$  $M^{(x)}(x;k)=I+O(\frac{1}{k}),~k\rightarrow\infty$;

$\bullet$ We assume that $a(k)$ has $n$ simple zeros $\{k_j\}_{1}^{n}$ in $\bfC_+$. The associated residues of $M^{(x)}(x;k)$ satisfy the relations
\bea
{Res}_{k=k_j}[M^{(x)}(x;k)]_1&=&\frac{e^{2\ii k_jx}}{\dot{a}(k_j)b(k_j)}
[M^{(x)}(x;k_j)]_2;\label{2.35}\\
{Res}_{k=\bar{k}_j}[M^{(x)}(x;k)]_2&=&\frac{e^{-2\ii\bar{k}_jx}}
{\overline{\dot{a}(k_j)b(k_j)}}[M^{(x)}(x;\bar{k}_j)]_1;\label{2.36}
\eea

(v) $\mathcal{J}^{-1}=\mathcal{K}$.

We refer to the appendix of \cite{JLe} for a derivation of (iv) and (v) in the similar
case of the derivative NLS equation.

Let $g_0(t),g_1(t)$ be smooth functions, we define the map
\berr
\mathcal{\tilde{J}}:\{g_0(t),g_1(t)\}\mapsto \{A(k),B(k)\}
\eerr
as follows
\be\label{2.37}
\begin{pmatrix}
B(k)\\
A(k)\\\end{pmatrix}=[\mu_1(0;k)]_2,\quad \mbox{Im}k^2\geq0,
\ee
where
$$\mu_1(t;k)=I+\int^t_T\e^{2\ii k^2(t'-t)\hat{\sigma}_3}
\big(U_2(0,t';k)\mu_1(t';k)\big)\dd t'$$
with $U_2(0,t;k)$ is determined by \eqref{2.23}.

The spectral function $A(k)$ and $B(k)$ have the following properties:

(i) $A(k)$ and $B(k)$ are entire function and bounded for $k\in\bar{D}_1\cup \bar{D}_3$; if $T=\infty$, $A(k)$ and $B(k)$ are defined only in $\bar{D}_1\cup \bar{D}_3$.

(ii) $A(k)=1+O(\frac{1}{k})$, $B(k)=O(\frac{-\frac{\ii}{2}g_0(0)}{k})$, $k\rightarrow\infty$;

(iii) $A(k)\overline{A(\bar{k})}-B(k)\overline{B(\bar{k})}=1$;

(iv) The map $\mathcal{\tilde{K}}:\{A(k),B(k)\}\mapsto \{g_0(t),g_1(t)\}$, inverse to $\mathcal{\tilde{J}}$, is defined as follows:
\bea\label{2.38}
\begin{aligned}
g_0(t)&=2\ii (M^{(t)}_1)_{12}(t)\e^{2\ii\int_0^t\Delta_2(t')\dd t'},\\
g_1(t)&=4(M^{(t)}_2)_{12}(t)\e^{2\ii\int_0^t\Delta_2(t')\dd t'}+2\ii g_0(t)\bigg((M^{(t)}_1)_{22}(t)-\beta|g_0(t)|^2\bigg),
\end{aligned}
\eea
where
\berr
\Delta_2(t)=16\beta\bigg(|(M^{(t)}_1)_{12}|^2\text{Re}(M^{(t)}_1)_{22}-\text{Re}\big[(M^{(t)}_1)_{12}(\bar{M}^{(t)}_2)_{12}\big]\bigg),
\eerr
and
\bea
M^{(t)}(t;k)&=&I+\frac{M^{(t)}_1(t)}{k}+\frac{M^{(t)}_2(t)}{k^2}
+O(\frac{1}{k^3}),~k\rightarrow\infty\nn
\eea
with $M^{(t)}(t;k)$ is the unique solution of the matrix  RH problem:

$\bullet$
$
M^{(t)}(t;k)=\left\{
\begin{aligned}
&M^{(t)}_+(t;k),\quad k\in D_1\cup D_3,\\
&M^{(t)}_-(t;k),\quad k\in D_2\cup D_4,
\end{aligned}
\right.
$
is a sectionally meromorphic function;

$\bullet$  $M^{(t)}(t;k)$ satisfies the jump condition
\be\label{2.39}
M^{(t)}_+(t;k)=M^{(t)}_-(t;k)J^{(t)}(t,k),~~\mbox{Im}k^2=0,
\ee
where
\be\label{2.40}
J^{(t)}(t,k)=\begin{pmatrix}
\frac{1}{A(k)\overline{A(\bar{k})}} ~& \frac{B(k)}{\overline{A(\bar{k})}}\e^{-4\ii k^2t} \\[4pt]
-\frac{\overline{B(\bar{k})}}{A(k)}\e^{4\ii k^2t} ~& 1 \\
\end{pmatrix};
\ee

$\bullet$  $M^{(t)}(t,k)=I+O(\frac{1}{k}),~k\rightarrow\infty$;

$\bullet$ We assume that $A(k)$ has $N$ simple zeros $\{z_j\}_{1}^{N}$ in $D_1\cup D_3$. The associated residues of $M^{(t)}(t;k)$ are given by
\bea
{Res}_{k=z_j}[M^{(t)}(t;k)]_1&=&\frac{e^{4\ii z_j^2t}}{\dot{A}(z_j)B(z_j)}
[M^{(t)}(t;z_j)]_2;\\
{Res}_{k=\bar{z}_j}[M^{(t)}(t;k)]_2&=&\frac{e^{-4\ii\bar{z}_j^2t}}
{\overline{\dot{A}(z_j)B(z_j)}}[M^{(t)}(t;\bar{z}_j)]_1;
\eea

(v) $\mathcal{\tilde{J}}^{-1}=\mathcal{\tilde{K}}$.

\section{The Riemann--Hilbert problem}
\setcounter{equation}{0}
\setcounter{lemma}{0}
\setcounter{theorem}{0}
Equations \eqref{2.17} and \eqref{2.18} can be rewritten in a form expressing the jump condition of a $2\times2$ RH problem. By solving this RH problem, the solution of KE equations \eqref{1.1} can be recovered. In fact, by using \eqref{2.17} and \eqref{2.18} and the definitions of the spectral functions \eqref{2.26}, we find
\be\label{3.1}
M_+(x,t;k)=M_-(x,t;k)J(x,t,k),\quad \mbox{Im}k^2=0,
\ee
where the matrices $M_-,M_+$ and $J$ are defined by
\bea\label{3.2}
M(x,t;k)=\left\{
\begin{aligned}
&\bigg(\frac{[\mu_2]_1}{a(k)}~[\mu_3]_2\bigg),\quad k\in D_1,\\
&\bigg(\frac{[\mu_1]_1}{d(k)}~[\mu_3]_2\bigg),\quad k\in D_2,\\
&\bigg([\mu_3]_1~\frac{[\mu_1]_2}{\overline{d(\bar{k})}}\bigg),\quad k\in D_3,\\
&\bigg([\mu_3]_1~\frac{[\mu_2]_2}{\overline{a(\bar{k})}}\bigg),\quad k\in D_4,
\end{aligned}
\right.
\eea
where
\be\label{3.3}
d(k)=a(k)\overline{A(\bar{k})}-b(k)\overline{B(\bar{k})},\quad k\in \bar{D}_2,
\ee
and
\bea\label{3.4}
&&J_1(x,t,k)=\begin{pmatrix}
1 & 0\\[4pt]
\Gamma(k)\e^{2\ii\theta(k)} & 1
\end{pmatrix},\nn\\
&&J_4(x,t,k)=\begin{pmatrix}
1-r_1(k)\overline{r_1(\bar{k})} ~& \overline{r_1(\bar{k})}\e^{-2\ii\theta(k)} \\[4pt]
-r_1(k)\e^{2\ii\theta(k)} ~& 1 \\
\end{pmatrix},\\
&&J_3(x,t,k)=\begin{pmatrix}
1 &~ -\overline{\Gamma(\bar{k})}\e^{-2\ii\theta(k)}\\[4pt]
0 & 1
\end{pmatrix},\nn\\
&&J_2(x,t,k)=(J_1J_4^{-1}J_3)(x,t,k)=\begin{pmatrix}
1 &~ -\overline{r(\bar{k})}\e^{-2\ii\theta(k)}\\[4pt]
r(k)\e^{2\ii\theta(k)} ~& 1-r(k)\overline{r(\bar{k})} \\
\end{pmatrix},\nn
\eea
\be\label{3.5}
r_1(k)=\frac{\overline{b(\bar{k})}}{a(k)},\quad\Gamma(k)=-\frac{\overline{B(\bar{k}})}{a(k)d(k)},\quad r(k)=r_1(k)+\Gamma(k),\quad\theta(k)=k x+2k^2t.
\ee

The contour for this RH problem is depicted in Fig. 2.
\begin{figure}[htbp]
  \centering
  \includegraphics[width=4in]{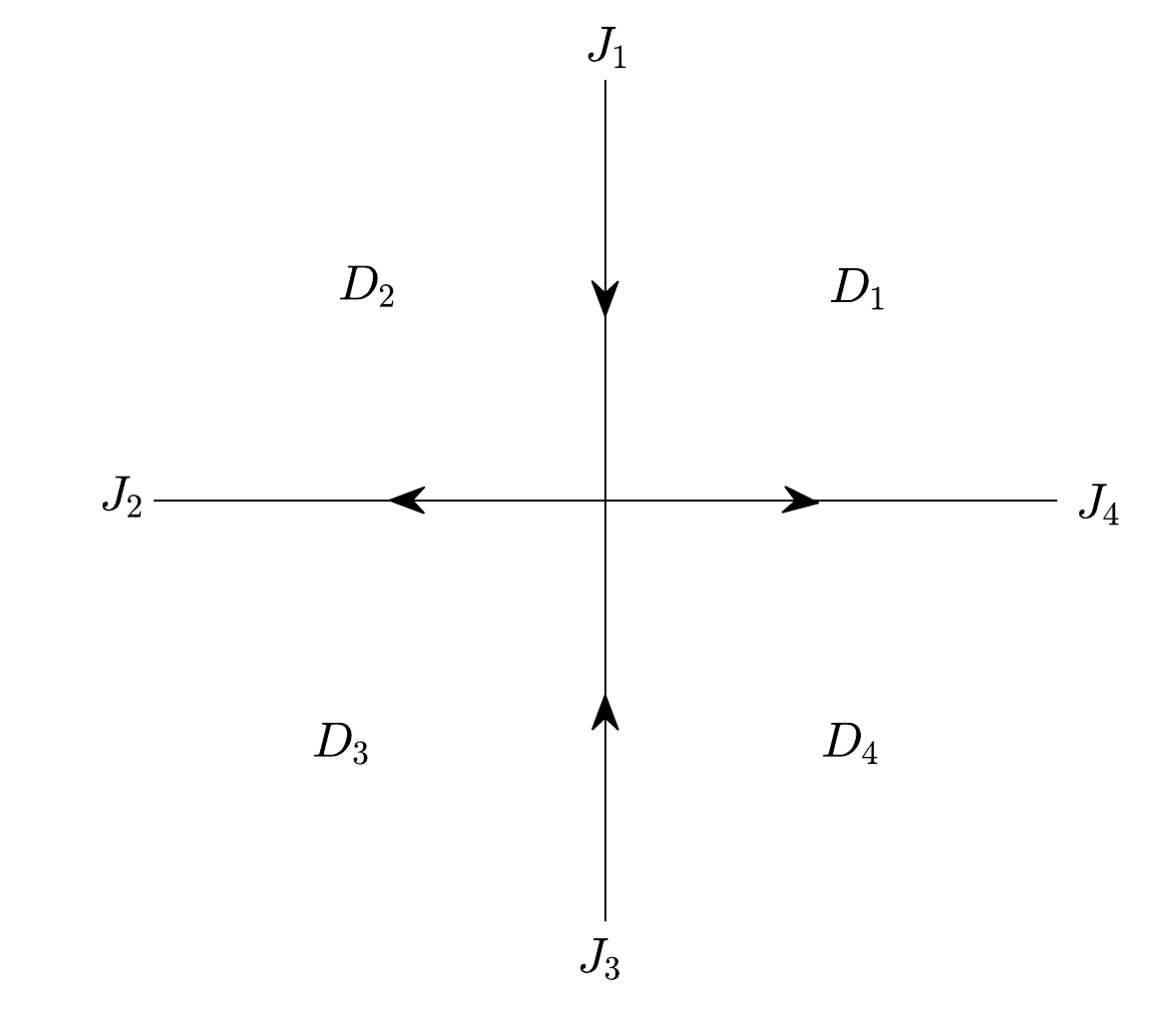}
  \caption{The contour for the RH problem.}
\end{figure}

The matrix function $M(x,t;k)$ defined by \eqref{3.2} is, in general, a sectionally meromorphic function of $k$, with possible poles at the zeros of $a(k)$, $d(k)$ and at the complex conjugates of these zeros. We have assumed that

$\bullet$ $a(k)$ can have $n=n_1+n_2$ simple zeros in $\{k\in\bfC|\mbox{Im}k>0\}$: $k_j\in D_1$, $j=1,\cdots,n_1$, $k_j\in D_2$, $j=n_1+1,\cdots,n$.

$\bullet$ $d(k)$ can have $\Lambda$ simple zeros: $\{\lambda_j\}_1^\Lambda\in D_2$;

$\bullet$ none of the zeros of $a(k)$ in $D_2$ coincide with a zero of $d(k)$;

$\bullet$ none of these functions has zeros on $\bfR\cup\ii\bfR$.\\
The associated residue formulaes are the following:
\bea\label{3.6}
\begin{aligned}
{Res}_{k=k_j}[M(x,t;k)]_1&=\frac{e^{2\ii\theta(k_j)}}{\dot{a}(k_j)b(k_j)}
[M(x,t;k_j)]_2,~j=1,\cdots,n_1;\\
{Res}_{k=\bar{k}_j}[M(x,t;k)]_2&=\frac{e^{-2\ii\theta(\bar{k}_j)}}
{\overline{\dot{a}(k_j)b(k_j)}}[M(x,t;\bar{k}_j)]_1,~j=1,\cdots,n_1;
\end{aligned}
\eea
and
\bea\label{3.7}
\begin{aligned}
{Res}_{k=\lambda_j}[M(x,t;k)]_1&=\frac{\overline{B(\bar{\lambda}_j)}
e^{2\ii\theta(\lambda_j)}}{\dot{d}(\lambda_j)a(\lambda_j)}[M(x,t;\lambda_j)]_2,~j=1,\cdots,\Lambda;\\
{Res}_{k=\bar{\lambda}_j}[M(x,t;k)]_2&=\frac{B(\bar{\lambda}_j)e^{-2\ii\theta(\bar{\lambda}_j)}}
{\overline{\dot{d}(\lambda_j)a(\lambda_j)}}[M(x,t;\bar{\lambda}_j)]_1,~j=1,\cdots,\Lambda;
\end{aligned}
\eea
where $\theta(k_j)=k_jx+2k_j^2t$.

The main result on the inverse spectral problem is the following.
\begin{theorem}\label{th3.1}
Let $u_0(x)\in \mathcal{S}([0,\infty))$, here $\mathcal{S}([0,\infty))$ is the space of Schwartz function on $[0,\infty)$, suppose that $u_0(x)$ is compatible with smooth functions $g_0(t)$ and $g_1(t)$ at $x=t=0$. The spectral functions $s(k)$ and $S(k)$ defined by \eqref{2.19} satisfy the global relation \eqref{2.29} for $T<\infty$ with $c^+(k)$ is analytic in $\{k\in\bfC|\mbox{Im}k>0\}$ and \eqref{2.30} for $T=\infty$ (in this case, it is assumed that $g_0(t), g_1(t)$ belong to $\mathcal{S}([0,\infty))$). The assumption of the  possible zeros for $a(k)$ and $d(k)$ are as stated above.

Let $M(x,t;k)$ as the solution of the following matrix RH problem:

$\bullet$ $M(x,t;k)$ is sectionally meromorphic in $\bfC\setminus\{\mbox{Im}k^2=0\}$,

$\bullet$  $M(x,t;k)$ satisfies the jump condition \eqref{3.1} with the jump matrices in \eqref{3.4},

$\bullet$  $M(x,t;k)$ has the following asymptotics:
\be\label{3.8}
M(x,t;k)=I+O\bigg(\frac{1}{k}\bigg),~k\rightarrow\infty,
\ee

$\bullet$ The associated residues of $M(x,t;k)$ satisfy the relations in \eqref{3.6}-\eqref{3.7}.

Then $M(x,t;k)$ exists and is unique and $u(x,t)$ is defined by
\bea\label{3.9}
u(x,t)&=&2\ii m(x,t)\e^{2\ii\int_{(0,0)}^{(x,t)}\Delta},\nn\\
m(x,t)&=&\lim_{k\rightarrow\infty}(k M(x,t;k))_{12},\\
\Delta(x,t)&=&-4\beta|m|^2\dd x+4\ii\beta(m\bar{m}_x-m_x\bar{m})\dd t,\nn
\eea
satisfies the KE equation \eqref{1.1}. Furthermore, $u(x,t)$ satisfies the initial and boundary conditions
\berr
u(x,0)=u_0(x),~u(0,t)=g_0(t),~u_x(0,t)=g_1(t).
\eerr
\end{theorem}
\begin{proof}
In the case when $a(k)$ and $d(k)$ have no zeros, the unique solvability is a consequence of an appropriate vanishing lemma as in the case of the NLS equation \cite{F3}. If $a(k)$ and $d(k)$ have zeros, this singular RH problem can be mapped to a regular one coupled with a system of algebraic equations \cite{FoI}. Moreover, it follows from standard arguments using the dressing method \cite{VEZ1,VEZ2} that if $M$ solves the above RH problem and $u(x,t)$ is defined by \eqref{3.9}, then $u(x,t)$ solves the KE equation \eqref{1.1}. The proof that $u(x,0)=u_0(x),~u(0,t)=g_0(t)$ and $u_{x}(0,t)=g_1(t)$ follows arguments similar to the ones used in \cite{JLe}.
\end{proof}

\section{Long-time asymptotics}
\setcounter{equation}{0}
\setcounter{lemma}{0}
\setcounter{theorem}{0}
In this section, we aim to derive the long-time asymptotics of the solution $u(x,t)$ of KE equation \eqref{1.1} on the half-line  based on a nonlinear steepest descent analysis of the associated RH problem established in Section 3. Now we fix $T=\infty$, that is, we are concerned with the IBV problem for equation \eqref{1.1} posed in the quarter-plane domain
$$\{(x,t)\in\bfR^2|x\geq0,t\geq0\}.$$ Furthermore, we make the following assumptions.
\begin{assumption}\label{ass1}
In what follows, we assume that the following conditions hold:

$\bullet$  the initial and boundary values lie in the Schwartz class.

$\bullet$ the spectral functions $a(k), b(k), A(k), B(k)$ defined in \eqref{2.26} satisfy the global relation \eqref{2.30}.

$\bullet$ $a(k)$ and $d(k)$ have no zeros in $\bar{D}_1\cup \bar{D}_2$ and $\bar{D}_2$, respectively.

$\bullet$ the initial and boundary values $u_0(x), g_0(t)$, and $g_1(t)$ are compatible with equation \eqref{1.1} to all orders at $x=t=0$, i.e., they satisfy
\bea
g_0(0)&=&u_0(0),\quad g_1(0)=u_0'(0),\quad \nn\\
\ii g'_0(0)+u_{0}''(0)-2|u_0(0)|^2u_0(0)&+&4\beta^2|u_0(0)|^4u_0(0)+4\ii\beta(|u_0(0)|^2)_xu_0(0)=0,\cdots.\nn
\eea
\end{assumption}

Let $\xi=\frac{x}{t}$, $\Sigma=\bfR\cup\ii\bfR$ and introduce a new phase function $\Phi(\xi,k)=4\ii k^2+2\ii\xi k$. Then we can rewrite the RH problem \eqref{3.1} as
\bea\label{4.1}
M_+(x,t;k)=M_-(x,t;k)J(x,t,k),\quad k\in\Sigma
\eea
with the jump matrix $J(x,t,k)$ is given by
\be\label{4.2}
J(x,t,k)=\left\{
\begin{aligned}
&\begin{pmatrix}
1 & 0\\[4pt]
\Gamma(k)\e^{t\Phi(\xi,k)} & 1
\end{pmatrix},\qquad\qquad\qquad\quad~ k\in\ii\bfR_+,\\
&\begin{pmatrix}
1 &~ -\overline{r(\bar{k})}\e^{-t\Phi(\xi,k)}\\[4pt]
r(k)\e^{t\Phi(\xi,k)} ~& 1-r(k)\overline{r(\bar{k})} \\
\end{pmatrix},\quad~ k\in\bfR_-,\\
&\begin{pmatrix}
1 &~ -\overline{\Gamma(\bar{k})}\e^{-t\Phi(\xi,k)}\\[4pt]
0 & 1
\end{pmatrix},\qquad\qquad\quad~ k\in\ii\bfR_-,\\
&\begin{pmatrix}
1-r_1(k)\overline{r_1(\bar{k})} ~& \overline{r_1(\bar{k})}\e^{-t\Phi(\xi,k)} \\[4pt]
-r_1(k)\e^{t\Phi(\xi,k)} ~& 1 \\
\end{pmatrix},~~ k\in\bfR_+,
\end{aligned}
\right.
\ee
where $\bfR_+=[0,\infty)$ and $\bfR_-=(-\infty,0]$ denote the positive and negative halves of the real axis.
Functions $r_1(k),\Gamma(k), r(k)$ are defined by
\bea\label{4.3}
r_1(k)&=&\frac{\overline{b(\bar{k})}}{a(k)},~~~~~\qquad k\in\bfR,\nn\\
\Gamma(k)&=&-\frac{\overline{B(\bar{k}})}{a(k)d(k)},~~~k\in\bar{D}_2,\\
 r(k)&=&r_1(k)+\Gamma(k),~k\in\bfR_-,\nn
\eea
and these functions admit the following properties from the discussion in Section 2:

$\bullet$ $r_1(k)$ is smooth and bounded on $\bfR$;

$\bullet$ $\Gamma(k)$ is smooth and bounded on $\bar{D}_2$ and analytic in $D_2$;

$\bullet$ $r(k)$ is smooth and bounded on $(-\infty, 0]$ and $\sup_{k\in(-\infty, 0]}|r(k)|<1$;

$\bullet$ There exist complex constants $\{\Gamma_j\}_{j=1}^\infty$ and $\{r_j\}_{j=1}^\infty$ such that, for any $N\geq1$,
\bea
\Gamma(k)&=&\sum_{j=1}^N\frac{\Gamma_j}{k^j}+O\bigg(\frac{1}{k^{N+1}}\bigg),~~k\rightarrow\infty,~~k\in\bar{D}_2,\label{4.4}\\
r(k)&=&\sum_{j=1}^N\frac{r_j}{k^j}+O\bigg(\frac{1}{k^{N+1}}\bigg),~~|k|\rightarrow\infty,~~k\in\bfR_-.\label{4.5}
\eea
\begin{remark}
We are in a position to prove $\sup_{k\in(-\infty, 0]}|r(k)|<1$. Recalling the definitions $s(k)$ and $S(k)$ in \eqref{2.26}, we have
\berr
S^{-1}(k)s(k)=\begin{pmatrix}
\overline{d(\bar{k})} & c(k)\\[4pt]
\overline{c(\bar{k})} & d(k)
\end{pmatrix}
\eerr
with $c(k)=b(k)A(k)-a(k)B(k)$. However, $r(k)=r_1(k)+\Gamma(k)=\frac{\overline{c(\bar{k})}}{d(k)}$, hence, the result $\sup_{k\in(-\infty, 0]}|r(k)|<1$ is a consequence of $\det s(k)=\det S(k)=1$.
\end{remark}
\subsection{Transformations of the RH problem}
By performing a number of transformations, we can bring the RH problem \eqref{4.1} to a form suitable for determining the long-time asymptotics. Let $N>1$ be given, and let $\mathcal{I}$ denote the interval $\mathcal{I}=(0,N]$. The jump matrix $J$ defined in \eqref{4.2} involves the exponentials $\e^{\pm t\Phi}$, where $\Phi(\xi,k)$ is defined by
$$\Phi(\xi,k)=4\ii k^2+2\ii\xi k,\quad \xi\in\mathcal{I},\quad k\in\bfC.$$
It follows that there is a single stationary point located at the point where $\frac{\partial\Phi}{\partial k}=0$,
i.e., at $k=k_0=-\frac{\xi}{4}$.

\begin{figure}[htbp]
  \centering
  \includegraphics[width=4in]{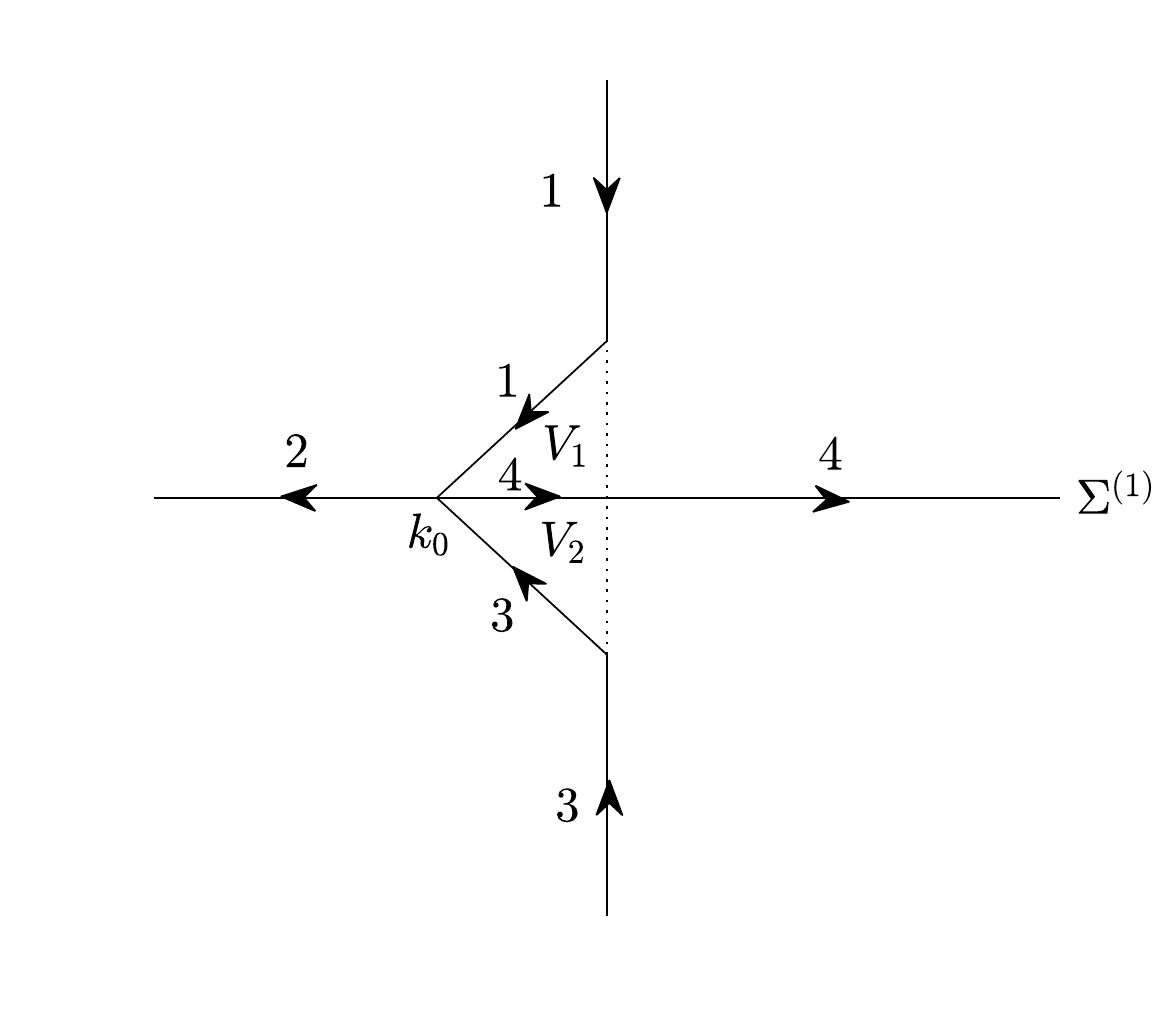}
  \caption{The contour $\Sigma^{(1)}$ in the complex $k$-plane.}\label{fig3}
\end{figure}

The first transformation is to deform the vertical part of $\Sigma$ so that it passes through the critical point $k_0$. Letting $V_1$ and $V_2$ denote the triangular domains shown in Fig. \ref{fig3}. The first transform is :
\be\label{4.6}
M^{(1)}(x,t;k)=M(x,t;k)\times\left\{
\begin{aligned}
&\begin{pmatrix}
1 & 0\\[4pt]
\Gamma(k)\e^{t\Phi(\xi,k)} & 1
\end{pmatrix},~~~~k\in V_1,\\
&\begin{pmatrix}
1 &~ \overline{\Gamma(\bar{k})}\e^{-t\Phi(\xi,k)}\\[4pt]
0 & 1
\end{pmatrix},~~k\in V_2,\\
&I,\qquad\qquad\qquad\qquad~~ \text{elsewhere}.
\end{aligned}
\right.
\ee
Then we obtain the RH problem
\be\label{4.7}
M^{(1)}_+(x,t;k)=M^{(1)}_-(x,t;k)J^{(1)}(x,t,k)
\ee
on the contour $\Sigma^{(1)}$ depicted in Fig. \ref{fig3}. The jump matrix $J^{(1)}(x,t,k)$ is given by
\bea
J^{(1)}_1&=&\begin{pmatrix}
1 & 0\\[4pt]
\Gamma\e^{t\Phi} & 1
\end{pmatrix},\qquad J^{(1)}_2=\begin{pmatrix}
1 &~ -\bar{r}\e^{-t\Phi}\\[4pt]
r\e^{t\Phi} ~& 1-r\bar{r} \\
\end{pmatrix},\nn\\
J^{(1)}_3&=&\begin{pmatrix}
1 &~ -\bar{\Gamma}\e^{-t\Phi}\\[4pt]
0 & 1
\end{pmatrix}, J^{(1)}_4=\begin{pmatrix}
1-r_1\bar{r}_1 ~& \bar{r}_1\e^{-t\Phi} \\[4pt]
-r_1\e^{t\Phi} ~& 1 \\
\end{pmatrix},\nn
\eea
where $J^{(1)}_i$ denotes the restriction of $J^{(1)}$ to the contour labeled by $i$ in Fig. \ref{fig3}.

The next transformation is:
\be\label{4.8}
M^{(2)}(x,t;k)=M^{(1)}(x,t;k)(x,t;k)\delta^{-\sigma_3}(\xi,k),
\ee
where the complex-valued function $\delta(\xi,k)$ is given by
\be\label{4.9}
\delta(\xi,k)=\exp\bigg\{\frac{1}{2\pi\ii}\int_{-\infty}^{k_0}\frac{\ln(1-|r(s)|^2)}{s-k}\dd s\bigg\},\quad k\in\bfC\setminus(-\infty,k_0].
\ee
The function $\delta$ satisfies the following jump condition across the real axis:
\berr
\delta_+(\xi,k)=\left\{
\begin{aligned}
&\frac{\delta_-(\xi,k)}{1-|r(k)|^2},~~k<k_0,\\
&\delta_-(\xi,k),\qquad k>k_0.
\end{aligned}
\right.
\eerr
Moreover, $\delta(\xi,k)\rightarrow1$ as $k\rightarrow\infty$, $\delta(\xi,k)$ and $\delta^{-1}(\xi,k)$ are bounded and analytic functions of $k\in\bfC\setminus(-\infty,k_0]$ with continuous boundary values on $(-\infty,k_0]$.

Then $M^{(2)}(x,t;k)$ satisfies the following RH problem
\be\label{4.10}
M^{(2)}_+(x,t;k)=M^{(2)}_-(x,t;k)J^{(2)}(x,t,k)
\ee
with the contour $\Sigma^{(2)}=\Sigma^{(1)}$ and the jump matrix $J^{(2)}=\delta_-^{\sigma_3}J^{(1)}\delta_+^{-\sigma_3}$, namely,
\bea
J^{(2)}_1&=&\begin{pmatrix}
1 & 0\\[4pt]
\Gamma\delta^{-2}\e^{t\Phi} & 1
\end{pmatrix},\qquad J^{(2)}_2=\begin{pmatrix}
1 ~& -r_2\delta_-^{2}\e^{-t\Phi} \\[4pt]
0 ~& 1 \\
\end{pmatrix}\begin{pmatrix}
1 ~& 0 \\[4pt]
\bar{r}_2\delta_+^{-2}\e^{t\Phi} ~& 1 \\
\end{pmatrix},\nn\\
J^{(2)}_3&=&\begin{pmatrix}
1 &~ -\bar{\Gamma}\delta^{2}\e^{-t\Phi}\\[4pt]
0 & 1
\end{pmatrix}, J^{(2)}_4=\begin{pmatrix}
1 ~& \bar{r}_1\delta^{2}\e^{-t\Phi} \\[4pt]
0 ~& 1 \\
\end{pmatrix}\begin{pmatrix}
1 ~& 0 \\[4pt]
-r_1\delta^{-2}\e^{t\Phi} ~& 1 \\
\end{pmatrix},\nn
\eea
where we define $r_2(k)$ by
\be\label{4.11}
r_2(k)=\frac{\overline{r(\bar{k})}}{1-r(k)\overline{r(\bar{k})}}.
\ee

Before processing the next deformation, we follow the idea of \cite{JL1,JL2} and decompose each of the functions $\Gamma$, $r_1$, $r_2$ into an analytic part and a small remainder because the spectral functions have limited domains of analyticity. The analytic part of the jump matrix will be deformed, whereas the small remainder will be left on the original contour. In face, we have the following lemmas.
\begin{lemma}\label{lem1}
There exist a decomposition
\berr
\Gamma(k)=\Gamma_a(t,k)+\Gamma_r(t,k),\quad t>0,~~k\in\ii\bfR_+,
\eerr
where the functions $\Gamma_a$ and $\Gamma_r$ have the following properties:

(1) For each $t>0$, $\Gamma_a(t,k)$ is defined and continuous for $k\in\bar{D}_1$ and analytic for $k\in D_1.$

(2) For each $\xi\in\mathcal{I}$ and each $t>0$, the function $\Gamma_a(t,k)$ satisfies
\be\label{4.12}
\begin{aligned}
|\Gamma_a(t,k)-\Gamma(0)|&\leq C\e^{\frac{t}{4}|\text{Re}\Phi(\xi,k)|},\\
|\Gamma_a(t,k)|&\leq \frac{C}{1+|k|}\e^{\frac{t}{4}|\text{Re}\Phi(\xi,k)|},
\end{aligned}
\ee
for $k\in\bar{D}_1$, where the constant $C$ is independent of $\xi, k, t$.

(3) The $L^1, L^2$ and $L^\infty$ norms of the function $\Gamma_r(t,\cdot)$ on $\ii\bfR_+$ are $O(t^{-3/2})$ as $t\rightarrow\infty$.
\end{lemma}
\begin{proof}
Since $\Gamma(k)\in C^5(\ii\bfR_+)$, then there exist complex constants $\{p_j\}_0^4$ such that
\be\label{4.13}
\Gamma^{(n)}(k)=\frac{\dd^n}{\dd k^n}\bigg(\sum_{j=0}^4p_jk^j\bigg)+O(k^{5-n}),\quad k\rightarrow0,~k\in\ii\bfR_+,~n=0,1,2.
\ee
In fact, $p_j=\frac{\Gamma^{(j)}(0)}{j!}$ for $j=0,1,\ldots,4$. On the other hand, we have
\be\label{4.14}
\Gamma^{(n)}(k)=\frac{\dd^n}{\dd k^n}\bigg(\sum_{j=1}^3\Gamma_jk^{-j}\bigg)+O(k^{-4-n}),\quad k\rightarrow\infty,~k\in\ii\bfR_+,~n=0,1,2.
\ee
Let
\be\label{4.15}
f_0(k)=\sum_{j=1}^8\frac{a_j}{(k+\ii)^j},
\ee
where $\{a_j\}_1^8$ are complex such that
\be\label{4.16}
f_0(k)=\left\{
\begin{aligned}
&\sum_{j=0}^4p_jk^j+O(k^5),\qquad ~k\rightarrow0,\\
&\sum_{j=1}^3\Gamma_jk^{-j}+O(k^{-4}),\quad k\rightarrow\infty.
\end{aligned}
\right.
\ee
It is easy to verify that \eqref{4.16} imposes eight linearly independent conditions on the $a_j$, hence the coefficients $a_j$ exist and are unique. Letting $f=\Gamma-f_0$, it follows that \\
(i) $f_0(k)$ is a rational function of $k\in\bfC$ with no poles in $\bar{D}_1$;\\
(ii) $f_0(k)$ coincides with $\Gamma(k)$ to four order at 0 and to third order at $\infty$, more precisely,
\be\label{4.17}
\frac{\dd^n}{\dd k^n}f(k)=\left\{
\begin{aligned}
&O(k^{5-n}),~\quad k\rightarrow0,\\
&O(k^{-4-n}),~~ k\rightarrow\infty,
\end{aligned}
\quad k\in\ii\bfR_+,~n=0,1,2.
\right.
\ee
The decomposition of $\Gamma(k)$ can be derived as follows. The map $k\mapsto\psi=\psi(k)$ defined by $\psi(k)=4k^2$ is a bijection $[0,\ii\infty)\mapsto(-\infty,0]$, so we may define a function $F$ : $\bfR\rightarrow\bfC$ by
\be\label{4.18}
F(\psi)=\left\{
\begin{aligned}
&(k+\ii)^2f(k),~\psi\leq0,\\
&0,~~\qquad\qquad~ \psi>0.
\end{aligned}
\right.
\ee
$F(\psi)$ is $C^5$ for $\psi\neq0$ and
\berr
F^{(n)}(\psi)=\bigg(\frac{1}{8k}\frac{\partial}{\partial k}\bigg)^n\bigg((k+\ii)^2f(k)\bigg),~~\psi\leq0.
\eerr
By \eqref{4.17}, $F\in C^2(\bfR)$ and $F^{(n)}(\psi)=O(|\psi|^{-1-n})$ as $|\psi|\rightarrow\infty$ for $n=0,1,2$. In particular,
\be\label{4.19}
\bigg\|\frac{\dd^nF}{\dd\psi^n}\bigg\|_{L^2(\bfR)}<\infty,\quad n=0,1,2,
\ee
that is, $F$ belongs to $H^2(\bfR)$. By the Fourier transform $\hat{F}(s)$ defined by
\be\label{4.20}
\hat{F}(s)=\frac{1}{2\pi}\int_\bfR F(\psi)\e^{-\ii\psi s}\dd\psi
\ee
where
\be\label{4.21}
F(\psi)=\int_\bfR\hat{F}(s)\e^{\ii\psi s}\dd s,
\ee
it follows from Plancherel theorem that $\|s^2\hat{F}(s)\|_{L^2(\bfR)}<\infty$. Equations \eqref{4.18} and \eqref{4.21} imply
\be
f(k)=\frac{1}{(k+\ii)^2}\int_\bfR\hat{F}(s)\e^{\ii\psi s}\dd s,\quad k\in\ii\bfR_+.
\ee
Writing
$$f(k)=f_a(t,k)+f_r(t,k),\quad t>0,~k\in\ii\bfR_+,$$
where the functions $f_a$ and $f_r$ are defined by
\bea
f_a(t,k)&=&\frac{1}{(k+\ii)^2}\int_{-\frac{t}{4}}^\infty\hat{F}(s)\e^{4\ii k^2s}\dd s,\quad t>0,~k\in\bar{D}_1,\\
f_r(t,k)&=&\frac{1}{(k+\ii)^2}\int^{-\frac{t}{4}}_{-\infty}\hat{F}(s)\e^{4\ii k^2s}\dd s,\quad t>0,~k\in\ii\bfR_+,
\eea
we infer that $f_a(t,\cdot)$ is continuous in $\bar{D}_1$ and analytic in $D_1$. Moreover, since $|\text{Re}4\ii k^2|\leq|\text{Re}\Phi(\xi,k)|$ for $k\in\bar{D}_1$ and $\xi\in\mathcal{I}$, we can get
\bea\label{4.25}
|f_a(t,k)|&\leq&\frac{1}{|k+\ii|^2}\|\hat{F}(s)\|_{L^1(\bfR)}\sup_{s\geq-\frac{t}{4}}\e^{s\text{Re}4\ii k^2}\nn\\
&\leq&\frac{C}{1+|k|^2}\e^{\frac{t}{4}|\text{Re}\Phi(\xi,k)|},\quad t>0,~k\in\bar{D}_1,~\xi\in\mathcal{I}.
\eea
Furthermore, we have
\bea\label{4.26}
|f_r(t,k)|&\leq&\frac{1}{|k+\ii|^2}\int_{-\infty}^{-\frac{t}{4}}s^2|\hat{F}(s)|s^{-2}\dd s\nn\\
&\leq&\frac{C}{1+|k|^2}\|s^2\hat{F}(s)\|_{L^2(\bfR)}\sqrt{\int_{-\infty}^{-\frac{t}{4}}s^{-4}\dd s},\\
&\leq&\frac{C}{1+|k|^2}t^{-3/2},\quad t>0,~k\in\ii\bfR_+,~\xi\in\mathcal{I}.\nn
\eea
Hence, the $L^1,L^2$ and $L^\infty$ norms of $f_r$ on $\ii\bfR_+$ are $O(t^{-3/2})$. Letting
\bea
\Gamma_a(t,k)&=&f_0(k)+f_a(t,k),\quad t>0,~ k\in\bar{D}_1,\\
\Gamma_r(t,k)&=&f_r(t,k),~~\qquad\qquad t>0,~ k\in\ii\bfR_+,
\eea
we find a decomposition of $\Gamma$ with the properties listed in the statement of the lemma.
\end{proof}

We introduce the open subsets $\{\Omega_j\}_1^8$, as displayed in Fig. \ref{fig4}. The following lemma describes how to decompose $r_j$, $j=1,2$ into an analytic part $r_{j,a}$ and a small remainder $r_{j,r}$. A proof can be found in \cite{JL3}.

\begin{figure}[htbp]
  \centering
  \includegraphics[width=4in]{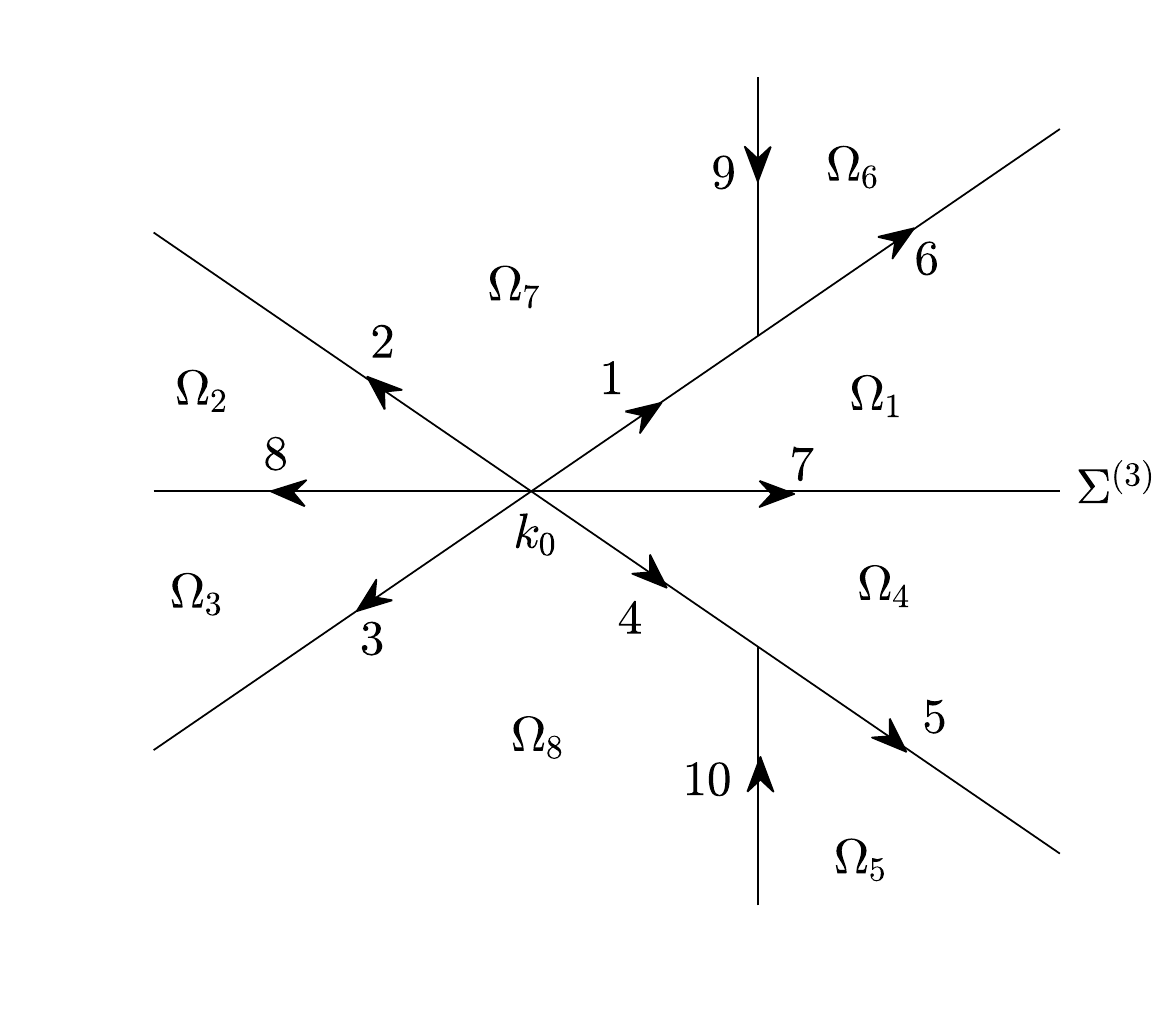}
  \caption{The contour $\Sigma^{(3)}$ and the open sets $\{\Omega_j\}_1^8$ in the complex $k$-plane.}\label{fig4}
\end{figure}
\begin{lemma}\label{lem2}
There exist decompositions
\be
\begin{aligned}
r_1(k)&=r_{1,a}(x,t,k)+r_{1,r}(x,t,k),\quad k>k_0,\\
r_2(k)&=r_{2,a}(x,t,k)+r_{2,r}(x,t,k),\quad k<k_0,
\end{aligned}
\ee
where the functions $\{r_{j,a}, r_{j,r}\}^2_1$ have the following properties:

(1) For each $\xi\in\mathcal{I}$ and each $t>0$, $r_{j,a}(x,t,k)$ is defined and continuous for $k\in\bar{\Omega}_j$ and
analytic for $\Omega_j$, $j= 1,2$.

(2) The functions $r_{1,a}$ and $r_{2,a}$ satisfy
\be\label{4.30}
\begin{aligned}
|r_{j,a}(x,t,k)-r_j(k_0)|&\leq C|k-k_0|\e^{\frac{t}{4}|\text{Re}\Phi(\xi,k)|},\\
|r_{j,a}(x,t,k)|&\leq \frac{C}{1+|k|}\e^{\frac{t}{4}|\text{Re}\Phi(\xi,k)|},
\end{aligned}
\quad t>0,~k\in\bar{\Omega}_j,~\xi\in\mathcal{I},~j=1,2,
\ee
where the constant $C$ is independent of $\xi, k, t$.

(3) The $L^1, L^2$ and $L^\infty$ norms of the function $r_{1,r}(x,t,\cdot)$ on $(k_0,\infty)$ are $O(t^{-3/2})$ as $t\rightarrow\infty$ uniformly with respect to $\xi\in\mathcal{I}$.

(4) The $L^1, L^2$ and $L^\infty$ norms of the function $r_{2,r}(x,t,\cdot)$ on $(-\infty,k_0)$ are $O(t^{-3/2})$ as $t\rightarrow\infty$ uniformly with respect to $\xi\in\mathcal{I}$.
\end{lemma}

The purpose of the next transformation is to deform the contour so that the jump matrix involves the exponential factor $\e^{-t\Phi}$ on the parts of the contour where Re$\Phi$ is positive and the factor $\e^{t\Phi}$ on the parts where Re$\Phi$ is negative according to the signature table for Re$\Phi$ as shown in Fig. \ref{fig5}. More precisely, we put
\be\label{4.31}
M^{(3)}(x,t;k)=M^{(2)}(x,t;k)G(k),
\ee
where
\be\label{4.32}
G(k)=\left\{
\begin{aligned}
&\begin{pmatrix}
1 &~ 0\\[4pt]
r_{1,a}\delta^{-2}\e^{t\Phi} ~& 1
\end{pmatrix},~\quad k\in\Omega_1,\\
&\begin{pmatrix}
1 ~& -r_{2,a}\delta^{2}\e^{-t\Phi}\\[4pt]
0 ~& 1
\end{pmatrix},\quad k\in\Omega_2,\\
&\begin{pmatrix}
1 ~& 0\\[4pt]
-\bar{r}_{2,a}\delta^{-2}\e^{t\Phi} ~& 1
\end{pmatrix},\quad k\in\Omega_3,\\
&\begin{pmatrix}
1 ~& \bar{r}_{1,a}\delta^2\e^{-t\Phi}\\[4pt]
0 ~& 1
\end{pmatrix},~~\quad k\in\Omega_4,\\
&\begin{pmatrix}
1 ~& -\bar{\Gamma}_{a}\delta^{2}\e^{-t\Phi} \\[4pt]
0 ~& 1
\end{pmatrix},~\quad k\in\Omega_5,\\
&\begin{pmatrix}
1 ~& 0\\[4pt]
-\Gamma_{a}\delta^{-2}\e^{t\Phi} ~& 1
\end{pmatrix},~\quad k\in\Omega_6,\\
&I,~~~~\qquad\qquad\qquad\quad k\in\Omega_7\cup\Omega_8.
\end{aligned}
\right.
\ee
Then the matrix $M^{(3)}(x,t;k)$ satisfies the following RH problem
\be\label{4.33}
M_+^{(3)}(x,t;k)=M_-^{(3)}(x,t;k)J^{(3)}(x,t,k),
\ee
with the jump matrix $J^{(3)}=G_+^{-1}(k)J^{(2)}G_-(k)$ is given by
\bea
J^{(3)}_1&=&\begin{pmatrix}
1 & 0\\[4pt]
-(r_{1,a}+\Gamma)\delta^{-2}\e^{t\Phi} & 1
\end{pmatrix},\quad J^{(3)}_2=\begin{pmatrix}
1 &~ -r_{2,a}\delta^2\e^{-t\Phi}\\
0~& 1
\end{pmatrix},\quad J^{(3)}_3=\begin{pmatrix}
1 &~ 0\\
\bar{r}_{2,a}\delta^{-2}\e^{t\Phi} ~& 1
\end{pmatrix},\nn\\
J^{(3)}_4&=&\begin{pmatrix}
1 &~ (\bar{r}_{1,a}+\bar{\Gamma})\delta^2\e^{-t\Phi}\\
0~& 1
\end{pmatrix},\quad~
J^{(3)}_5=\begin{pmatrix}
1 &~ (\bar{r}_{1,a}+\bar{\Gamma}_a)\delta^2\e^{-t\Phi}\\
0~& 1
\end{pmatrix},\nn\\
J^{(3)}_6&=&\begin{pmatrix}
1 &~ 0\\
-(r_{1,a}+\Gamma_a)\delta^{-2}\e^{t\Phi}~& 1
\end{pmatrix},\quad
J^{(3)}_7=\begin{pmatrix}
1 ~& \bar{r}_{1,r}\delta^{2}\e^{-t\Phi} \\[4pt]
0 ~& 1 \\
\end{pmatrix}\begin{pmatrix}
1 ~& 0 \\[4pt]
-r_{1,r}\delta^{-2}\e^{t\Phi} ~& 1 \\
\end{pmatrix},\nn\\
J^{(3)}_8&=&\begin{pmatrix}
1 ~& -r_{2,r}\delta_-^{2}\e^{-t\Phi} \\[4pt]
0 ~& 1 \\
\end{pmatrix}\begin{pmatrix}
1 ~& 0 \\[4pt]
\bar{r}_{2,r}\delta_+^{-2}\e^{t\Phi} ~& 1 \\
\end{pmatrix},\quad
J^{(3)}_9=\begin{pmatrix}
1 &~ 0\\[4pt]
\Gamma_r\delta^{-2}\e^{t\Phi} & 1
\end{pmatrix}, \quad J^{(3)}_{10}=\begin{pmatrix}
1 &~ -\bar{\Gamma}_r\delta^{2}\e^{-t\Phi}\\[4pt]
0 & 1
\end{pmatrix},\nn
\eea
with $J^{(3)}_i$ denoting the restriction of $J^{(3)}$ to the contour labeled by $i$ in Fig. \ref{fig4}.

\begin{figure}[htbp]
  \centering
  \includegraphics[width=3.3in]{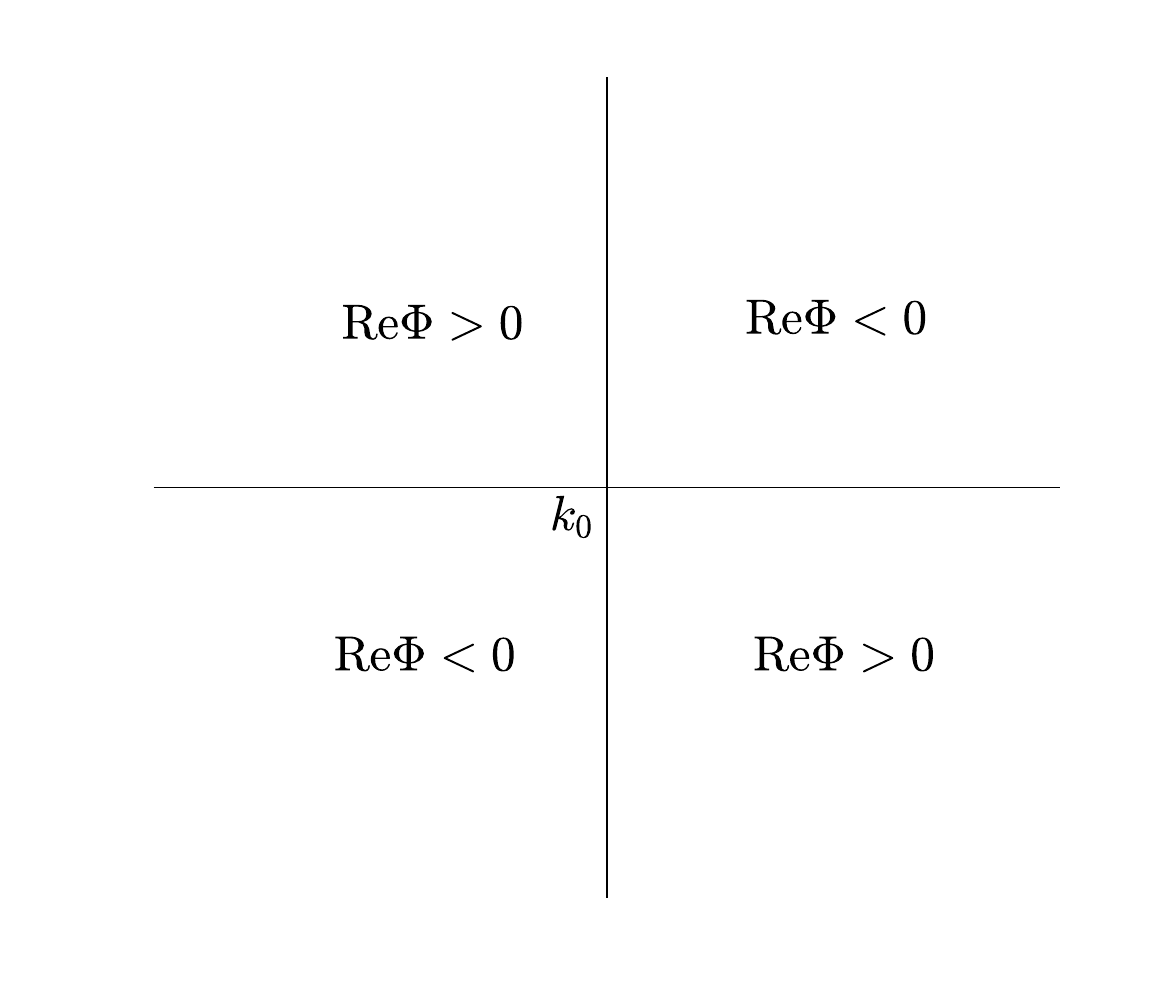}
  \caption{The signature table of Re$\Phi$.}\label{fig5}
\end{figure}

It is easy to check that the jump matrix $J^{(3)}$ decays to identity matrix $I$ as $t\rightarrow\infty$ everywhere except near $k_0$. Thus, the main contribution to the long-time asymptotics should come from a neighborhood of the stationary phase point $k_0$.

To focus on $k_0$, we make a local change of variables for $k$ near $k_0$ and introduce the new variable $z=z(\xi,k)$ by $$z=\sqrt{8t}(k-k_0).$$ Let $D_\varepsilon(k_0)$ denote the open disk of radius $\varepsilon$ centered at $k_0$ for a small $\varepsilon>0$. Then, the map $k\mapsto z$ is a bijection from $D_\varepsilon(k_0)$ to the open disk of radius $\sqrt{8t}\varepsilon$ centered at the origin for all $\xi\in\mathcal{I}$. Integrating by parts in formula \eqref{4.9} yields,
\be\label{4.34}
\delta(\xi,k)=\e^{\ii\nu\ln(k-k_0)}\e^{\chi(k)}=(k-k_0)^{\ii\nu}\e^{\chi(k)},
\ee
where
\bea
\nu&=&\nu(\xi)=-\frac{1}{2\pi}\ln(1-|r(k_0)|^2)>0,\label{4.35}\\
\chi(k)&=&-\frac{1}{2\pi\ii}\int_{-\infty}^{k_0}\ln(k-s)\dd\ln(1-|r(s)|^2).\label{4.36}
\eea
Hence we can write $\delta$ as
\berr
\delta(\xi,k)=z^{\ii\nu}\delta_0(\xi,t)\delta_1(\xi,k),
\eerr
with $\delta_0(\xi,t), \delta_1(\xi,k)$ defined by
\be\label{4.37}
\delta_0(\xi,t)=(8t)^{-\frac{\ii\nu}{2}}\e^{\chi(k_0)},\quad \delta_1(\xi,k)=\e^{\chi(k)-\chi(k_0)}.
\ee
Now we let
\be\label{4.38}
\tilde{M}(x,t;z)=M^{(3)}(x,t;k)\e^{-\frac{t\Phi(\xi,k_0)}{2}\sigma_3}\delta_0^{\sigma_3}(\xi,t),\quad k\in\bfC\setminus\Sigma^{(3)}.
\ee
Then $\tilde{M}$ is a sectionally analytic function of $z$ which satisfies
$$\tilde{M}_+(x,t;z)=\tilde{M}_-(x,t;z)\tilde{J}(x,t,z),\quad z\in X,$$
where the contour $X=X_1\cup X_2\cup X_3\cup X_4$ is the cross defined by
\be\label{4.39}
\begin{aligned}
X_1&=\{s\e^{\frac{\ii\pi}{4}}|0\leq s<\infty\},~~~ X_2=\{s\e^{\frac{3\ii\pi}{4}}|0\leq s<\infty\},\\
X_3&=\{s\e^{-\frac{3\ii\pi}{4}}|0\leq s<\infty\},~ X_4=\{s\e^{-\frac{\ii\pi}{4}}|0\leq s<\infty\},
\end{aligned}
\ee
and oriented as in Fig. \ref{fig6}.
\begin{figure}[htbp]
  \centering
  \includegraphics[width=3in]{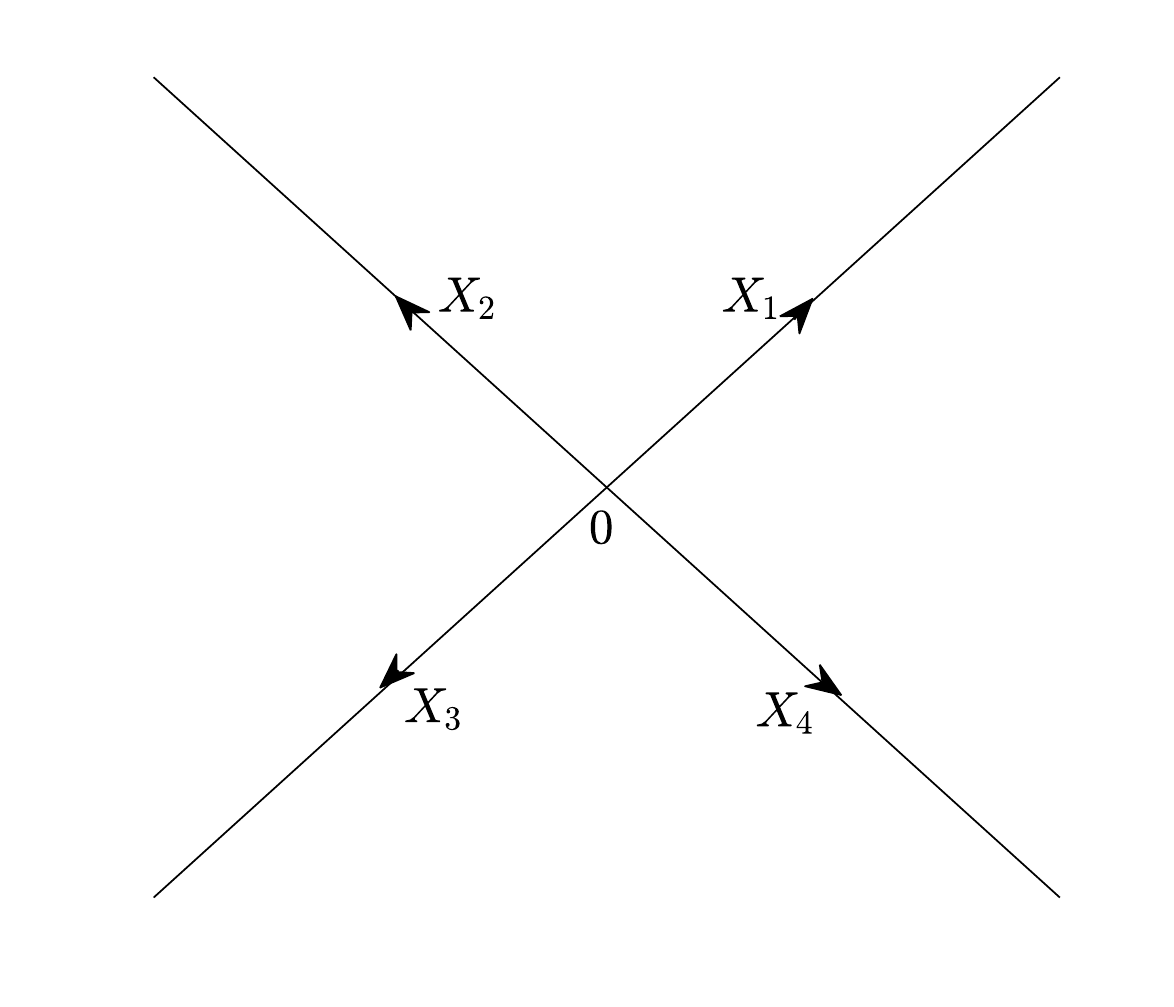}
  \caption{The contour $X=X_1\cup X_2\cup X_3\cup X_4$.}\label{fig6}
\end{figure}

The jump matrix
\be\label{4.40}
\tilde{J}(x,t,z)=\e^{\frac{t\Phi(\xi,k_0)}{2}\hat{\sigma}_3}\delta_0^{-\hat{\sigma}_3}(\xi,t)J^{(3)}(x,t,k)
\ee
is given by
\berr
\tilde{J}(x,t,z)=\left\{
\begin{aligned}
&\begin{pmatrix}
1 ~& 0\\[4pt]
-(r_{1,a}+\Gamma_a)\delta_1^{-2}\e^{\frac{\ii z^2}{2}}z^{-2\ii\nu} ~& 1
\end{pmatrix},\quad k\in \mathcal{X}_1\cap D_1,\\
&\begin{pmatrix}
1 ~& 0\\[4pt]
-(r_{1,a}+\Gamma)\delta_1^{-2}\e^{\frac{\ii z^2}{2}}z^{-2\ii\nu} ~& 1
\end{pmatrix},\quad k\in \mathcal{X}_1\cap D_2,\\
&\begin{pmatrix}
1 ~& -r_{2,a}\delta_1^{2}\e^{-\frac{\ii z^2}{2}}z^{2\ii\nu}\\[4pt]
0 ~& 1
\end{pmatrix},~\qquad\qquad k\in \mathcal{X}_2,\\
&\begin{pmatrix}
1 ~& 0\\[4pt]
\bar{r}_{2,a}\delta_1^{-2}\e^{\frac{\ii z^2}{2}}z^{-2\ii\nu} ~& 1
\end{pmatrix},~\qquad\qquad k\in \mathcal{X}_3,\\
&\begin{pmatrix}
1 ~& (\bar{r}_{1,a}+\bar{\Gamma})\delta_1^{2}\e^{-\frac{\ii z^2}{2}}z^{2\ii\nu}\\[4pt]
0 ~& 1
\end{pmatrix},~\quad\quad k\in \mathcal{X}_4\cap D_3,\\
&\begin{pmatrix}
1 ~& (\bar{r}_{1,a}+\bar{\Gamma}_a)\delta_1^{2}\e^{-\frac{\ii z^2}{2}}z^{2\ii\nu}\\[4pt]
0 ~& 1
\end{pmatrix},\quad\quad k\in \mathcal{X}_4\cap D_4,\\
\end{aligned}
\right.
\eerr
where $\mathcal{X}=X+k_0$ denote the cross $X$ centered at $k_0$ and we have used the following relation
$$t(\Phi(\xi,k)-\Phi(\xi,k_0))=4\ii t(k-k_0)^2=\frac{\ii z^2}{2}.$$

\subsection{Model RH problem}
Let $X=X_1\cup X_2\cup X_3\cup X_4\subset\bfC$ be the cross defined in \eqref{4.39} and oriented as in Fig. \ref{fig6}. Let $\mathcal{D}\subset\bfC$ denote the open unit disk, and also define the function $\nu:\mathcal{D}\rightarrow(0,\infty)$ by $\nu(q)=-\frac{1}{2\pi}\ln(1-|q|^2)$. We consider the following RH problems parametrized by $q\in\mathcal{D}$:
\be\label{4.41}\left\{
\begin{aligned}
&M^X_+(q,z)=M^X_-(q,z)J^X(q,z),~\text{for~almost~every}~z\in X,\\
&M^X(q,z)\rightarrow I,~~~~\qquad\qquad\qquad\text{as}~z\rightarrow\infty,
\end{aligned}
\right.
\ee
where the jump matrix $J^X(q,z)$ is defined by
\berr
J^X(q,z)=\left\{
\begin{aligned}
&\begin{pmatrix}
1 ~& 0\\[4pt]
-q\e^{\frac{\ii z^2}{2}}z^{-2\ii\nu} ~& 1
\end{pmatrix},~\quad\quad z\in X_1,\\
&\begin{pmatrix}
1 ~& -\frac{\bar{q}}{1-|q|^2}\e^{-\frac{\ii z^2}{2}}z^{2\ii\nu}\\[4pt]
0 ~& 1
\end{pmatrix},~~ z\in X_2,\\
&\begin{pmatrix}
1 ~& 0\\[4pt]
\frac{q}{1-|q|^2}\e^{\frac{\ii z^2}{2}}z^{-2\ii\nu} ~& 1
\end{pmatrix},~\quad z\in X_3,\\
&\begin{pmatrix}
1 ~& \bar{q}\e^{-\frac{\ii z^2}{2}}z^{2\ii\nu}\\[4pt]
0 ~& 1
\end{pmatrix},~\quad\qquad z\in X_4.
\end{aligned}
\right.
\eerr
Then the RH problem \eqref{4.41} can be solved explicitly in terms of parabolic cylinder functions \cite{PD,PD1}.
\begin{theorem}\label{th4.1}
The RH problem \eqref{4.41} has a unique solution $M^X(q,z)$ for each $q\in \mathcal{D}$. This solution satisfies
\be\label{4.42}
M^X(q,z)=I+\frac{\ii}{z}\begin{pmatrix}
0 ~& -\beta^X(q)\\[4pt]
\overline{\beta^X(q)} ~& 0
\end{pmatrix}+O\bigg(\frac{1}{z^2}\bigg),\quad z\rightarrow\infty,~q\in\mathcal{D},
\ee
where the error term is uniform with respect to $\arg z\in[0,2\pi]$ and the function $\beta^X(q)$ is given by
\be\label{4.43}
\beta^X(q)=\sqrt{\nu(q)}\e^{\ii\big(-\frac{3\pi}{4}-\arg q+\arg\Gamma(\ii\nu(q))\big)},\quad q\in\mathcal{D},
\ee
where $\Gamma(\cdot)$ denotes the standard Gamma function. Moreover, for each compact subset $\mathcal{D}'$ of $\mathcal{D}$,
\be\label{4.44}
\sup_{q\in\mathcal{D}'}\sup_{z\in\bfC\setminus X}|M^X(q,z)|<\infty.
\ee
\end{theorem}
\begin{proof}
The proof of this theorem can be analogously derived by the procedure used in \cite{PD,PD1,JL3}.
\end{proof}

Define $q=r(k_0)$, then $r(k)\rightarrow q$ and $\delta_1\rightarrow1$ as $t\rightarrow\infty$. This implies that the jump matrix $\tilde{J}$ tend to the matrix $J^X$ for large $t$. In other words, the jumps of $M^{(3)}$ for $k$ near $k_0$ approach those of the function $M^X\delta^{-\sigma_3}_0\e^{\frac{t\Phi(\xi,k_0)}{2}\sigma_3}$ as $t\rightarrow\infty$. Therefore, we can approximate $M^{(3)}$ in the neighborhood $D_\varepsilon(k_0)$ of $k_0$ by
\be\label{4.45}
M^{(k_0)}(x,t;k)=\e^{-\frac{t\Phi(\xi,k_0)}{2}\sigma_3}\delta^{\sigma_3}_0M^X(q,z)\delta^{-\sigma_3}_0\e^{\frac{t\Phi(\xi,k_0)}{2}\sigma_3}.
\ee

Let $\mathcal{X}^\varepsilon$ denote the part of $\mathcal{X}$  that lies in the disk $D_\varepsilon(k_0)$, that is, $\mathcal{X}^\varepsilon=\mathcal{X}\cap D_\varepsilon(k_0)$. Then we have the following lemma about the function $M^{(k_0)}$, which will be very useful in deriving the error bound in next section.
\begin{lemma}\label{lem3}
For each $t>0$ and $\xi\in\mathcal{I}$, the function $M^{(k_0)}(x,t;k)$ defined in \eqref{4.45} is an analytic function of $k\in D_\varepsilon(k_0)\setminus\mathcal{X}^\varepsilon$. Furthermore,
\be\label{4.46}
|M^{(k_0)}(x,t;k)|\leq C,\quad t>3,~\xi\in\mathcal{I},~k\in\overline{ D_\varepsilon(k_0)}\setminus\mathcal{X}^\varepsilon.
\ee
On the other hand, across $\mathcal{X}^\varepsilon$, $M^{(k_0)}$ satisfied the jump condition $M_+^{(k_0)}=M_-^{(k_0)}J^{(k_0)}$ with jump matrix $$J^{(k_0)}=\e^{-\frac{t\Phi(\xi,k_0)}{2}\hat{\sigma}_3}\delta^{\hat{\sigma}_3}_0J^X,$$ and $J^{(k_0)}$ satisfies the following estimates:
\be\label{4.47}
\left\{
\begin{aligned}
&\|J^{(3)}-J^{(k_0)}\|_{L^1(\mathcal{X}^\varepsilon)}\leq Ct^{-1}\ln t,\\
&\|J^{(3)}-J^{(k_0)}\|_{L^2(\mathcal{X}^\varepsilon)}\leq Ct^{-3/4}\ln t,\\
&\|J^{(3)}-J^{(k_0)}\|_{L^\infty(\mathcal{X}^\varepsilon)}\leq Ct^{-1/2}\ln t,
\end{aligned}
\quad t>3,~\xi\in\mathcal{I},
\right.
\ee
where $C>0$ is a constant independent of $t,\xi,k$. Moreover, as $t\rightarrow\infty$,
\be\label{4.48}
\|(M^{(k_0)})^{-1}(x,t;k)-I\|_{L^\infty(\partial D_\varepsilon(k_0))}=O(t^{-1/2}),
\ee
and
\be\label{4.49}
\frac{1}{2\pi\ii}\int_{\partial D_\varepsilon(k_0)}((M^{(k_0)})^{-1}(x,t;k)-I)\dd k=-\frac{\e^{-\frac{t\Phi(\xi,k_0)}{2}\hat{\sigma}_3}\delta^{\hat{\sigma}_3}_0M^X_1(\xi)}
{\sqrt{8t}}+O(t^{-1}),
\ee
where $M^X_1(\xi)$ is defined by
\be\label{4.50}
M^X_1(\xi)=\ii\begin{pmatrix}
0 ~& -\beta^X(q)\\[4pt]
\overline{\beta^X(q)} ~& 0
\end{pmatrix}.
\ee
\end{lemma}
\begin{proof}
The analyticity of $M^{(k_0)}$ is obvious. Since $|\e^{-\frac{t\Phi(\xi,k_0)}{2}}|=|\delta_0(\xi,t)|=1$, thus, the estimate \eqref{4.46} follows from the definition of $M^{(k_0)}$ in \eqref{4.45} and the estimate \eqref{4.44}.

A careful computation as in \cite{JL1,JL3} shows that
\be\label{4.51}
|\chi(k)-\chi(k_0)|\leq C|k-k_0|(1+\big|\ln|k-k_0|\big|),\quad \xi\in\mathcal{I},~k\in\mathcal{X}^\varepsilon.
\ee
Therefore, we have
\be\label{4.52}
|\delta_1(\xi,k)-1|=|\e^{\chi(k)-\chi(k_0)}-1|\leq C|k-k_0|(1+\big|\ln|k-k_0|\big|),\quad \xi\in\mathcal{I},~k\in\mathcal{X}^\varepsilon.
\ee
On the other hand, for $k\in\mathcal{X}_1^\varepsilon\cap D_1$, by the estimates \eqref{4.12} and \eqref{4.30}, we can find
\bea
|r_{1,a}(x,t,k)+\Gamma_a(t,k)-q|&=&|r_{1,a}(x,t,k)+\Gamma_a(t,k)-r_1(k_0)-\Gamma(k_0)|\nn\\
&\leq&|r_{1,a}(x,t,k)-r_1(k_0)|+|\Gamma_a(t,k)-\Gamma(0)|+|\Gamma(k_0)-\Gamma(0)|\nn\\
&\leq&C|k-k_0|\e^{\frac{t}{4}|\text{Re}\Phi(\xi,k)|}.\nn
\eea
Then, we obtain
\bea
|(r_{1,a}+\Gamma_a)\delta_1^{-2}-q|&\leq&|\delta_1^{-2}-1||r_{1,a}+\Gamma_a|+|r_{1,a}+\Gamma_a-q|\nn\\
&\leq&C|k-k_0|(1+\big|\ln|k-k_0|\big|)\e^{\frac{t}{4}|\text{Re}\Phi(\xi,k)|},\label{4.53}
\eea
for $k\in\mathcal{X}_1^\varepsilon\cap D_1$.

 Accordingly, we have
\be\label{4.54}
\begin{aligned}
|(r_{1,a}+\Gamma)\delta_1^{-2}-q|&\leq C|k-k_0|(1+\big|\ln|k-k_0|\big|)\e^{\frac{t}{4}|\text{Re}\Phi(\xi,k)|},~~k\in\mathcal{X}_1^\varepsilon\cap D_2,\\
\bigg|r_{2,a}\delta_1^{2}-\frac{\bar{q}}{1-|q|^2}\bigg|&\leq C|k-k_0|(1+\big|\ln|k-k_0|\big|)\e^{\frac{t}{4}|\text{Re}\Phi(\xi,k)|},~~k\in\mathcal{X}_2^\varepsilon,\\
\bigg|\bar{r}_{2,a}\delta_1^{-2}-\frac{q}{1-|q|^2}\bigg|&\leq C|k-k_0|(1+\big|\ln|k-k_0|\big|)\e^{\frac{t}{4}|\text{Re}\Phi(\xi,k)|},~~k\in\mathcal{X}_3^\varepsilon,\\
|(\bar{r}_{1,a}+\bar{\Gamma})\delta_1^{2}-\bar{q}|&\leq C|k-k_0|(1+\big|\ln|k-k_0|\big|)\e^{\frac{t}{4}|\text{Re}\Phi(\xi,k)|},~~k\in\mathcal{X}_4^\varepsilon\cap D_3,\\
|(\bar{r}_{1,a}+\bar{\Gamma}_a)\delta_1^{2}-\bar{q}|&\leq C|k-k_0|(1+\big|\ln|k-k_0|\big|)\e^{\frac{t}{4}|\text{Re}\Phi(\xi,k)|},~~k\in\mathcal{X}_4^\varepsilon\cap D_4.
\end{aligned}
\ee
Since
\berr
\tilde{J}-J^X=\left\{
\begin{aligned}
&\begin{pmatrix}
1 ~& 0\\[4pt]
-((r_{1,a}+\Gamma_a)\delta_1^{-2}-q)\e^{\frac{\ii z^2}{2}}z^{-2\ii\nu} ~& 1
\end{pmatrix},\quad k\in\mathcal{X}_1\cap D_1,\\
&\begin{pmatrix}
1 ~& 0\\[4pt]
-((r_{1,a}+\Gamma)\delta_1^{-2}-q)\e^{\frac{\ii z^2}{2}}z^{-2\ii\nu} ~& 1
\end{pmatrix},~~\quad k\in\mathcal{X}_1\cap D_2,\\
&\begin{pmatrix}
1 ~& -(r_{2,a}\delta_1^{2}-\frac{\bar{q}}{1-|q|^2})\e^{-\frac{\ii z^2}{2}}z^{2\ii\nu}\\[4pt]
0 ~& 1
\end{pmatrix},~~~\quad\quad k\in\mathcal{X}_2,\\
&\begin{pmatrix}
1 ~& 0\\[4pt]
(\bar{r}_{2,a}\delta_1^{-2}-\frac{q}{1-|q|^2})\e^{\frac{\ii z^2}{2}}z^{-2\ii\nu} ~& 1
\end{pmatrix},~\quad\quad~~~ k\in \mathcal{X}_3,\\
&\begin{pmatrix}
1 ~& ((\bar{r}_{1,a}+\bar{\Gamma})\delta_1^{2}-\bar{q})\e^{-\frac{\ii z^2}{2}}z^{2\ii\nu}\\[4pt]
0 ~& 1
\end{pmatrix},~\quad\quad ~~k\in \mathcal{X}_4\cap D_3,\\
&\begin{pmatrix}
1 ~& ((\bar{r}_{1,a}+\bar{\Gamma}_a)\delta_1^{2}-\bar{q})\e^{-\frac{\ii z^2}{2}}z^{2\ii\nu}\\[4pt]
0 ~& 1
\end{pmatrix},\quad\quad ~k\in \mathcal{X}_4\cap D_4.
\end{aligned}
\right.
\eerr
Then, \eqref{4.53}, \eqref{4.54} together with the facts
\bea
\text{Re}\bigg(\frac{\ii z^2}{2}\bigg)&=&t\text{Re}\Phi(\xi,k)=4t\text{Re}[\ii(k-k_0)^2]=-4t|k-k_0|^2,~~\quad \text{for~}k\in \mathcal{X}_1\cup \mathcal{X}_3,\nn\\
\text{Re}\bigg(-\frac{\ii z^2}{2}\bigg)&=&-t\text{Re}\Phi(\xi,k)=-4t\text{Re}[\ii(k-k_0)^2]=-4t|k-k_0|^2,\quad \text{for~}k\in \mathcal{X}_2\cup \mathcal{X}_4\nn
\eea
and $|z^{\pm2\ii\nu}|=\e^{\mp2\nu\arg z}$ is bounded by $0<\nu\leq\nu_{\max}=-(2\pi)^{-1}\ln(1-\sup_k|r(k)|^2)<\infty$ yield
\be\label{4.55}
|\tilde{J}-J^X|\leq C|k-k_0|(1+\big|\ln|k-k_0|\big|)\e^{-3t|k-k_0|^2},\quad k\in\mathcal{X}^\varepsilon.
\ee
However,
\berr
J^{(3)}-J^{(k_0)}=\e^{-\frac{t\Phi(\xi,k_0)}{2}\hat{\sigma}_3}\delta^{\hat{\sigma}_3}_0(\tilde{J}-J^X),\quad k\in\mathcal{X}^\varepsilon,
\eerr
hence, we immediately get
\be\label{4.56}
|J^{(3)}-J^{(k_0)}|\leq C|k-k_0|(1+\big|\ln|k-k_0|\big|)\e^{-3t|k-k_0|^2},\quad k\in\mathcal{X}^\varepsilon.
\ee
Thus, we conclude that
\berr
\|J^{(3)}-J^{(k_0)}\|_{L^\infty(\mathcal{X}^\varepsilon)}\leq C\sup_{0\leq s\leq\varepsilon}s(1+|\ln s|)\e^{-3ts^2}\leq Ct^{-1/2}\ln t,\quad t>3,~\xi\in\mathcal{I},
\eerr
\berr
\|J^{(3)}-J^{(k_0)}\|_{L^1(\mathcal{X}^\varepsilon)}\leq C\int_0^\varepsilon s(1+|\ln s|)\e^{-3ts^2}\dd s\leq Ct^{-1}\ln t,\quad t>3,~\xi\in\mathcal{I},
\eerr
and
\berr
\|J^{(3)}-J^{(k_0)}\|_{L^2(\mathcal{X}^\varepsilon)}\leq C\bigg(\int_0^\varepsilon \big(s(1+|\ln s|)\e^{-3ts^2}\big)^2\dd s\bigg)^{\frac{1}{2}}\leq Ct^{-3/4}\ln t,\quad t>3,~\xi\in\mathcal{I}.
\eerr

If $k\in\partial D_\varepsilon(k_0)$, the variable $z=\sqrt{8t}(k-k_0)$ tends to infinity as $t\rightarrow\infty$. It follows from \eqref{4.42} that
\berr
M^X(q,z)=I+\frac{M^X_1(\xi)}{\sqrt{8t}(k-k_0)}+O\bigg(\frac{1}{z^2}\bigg),\quad t\rightarrow\infty,~k\in \partial D_\varepsilon(k_0),
\eerr
where $M^X_1(\xi)$ is defined by \eqref{4.50}. Since the estimate \eqref{4.44} and $$M^{(k_0)}(x,t;k)=\e^{-\frac{t\Phi(\xi,k_0)}{2}\hat{\sigma}_3}\delta^{\hat{\sigma}_3}_0M^X(q,z),
~~|\e^{-\frac{t\Phi(\xi,k_0)}{2}}\delta_0|\leq C,$$
thus we have
\be\label{4.57}
(M^{(k_0)})^{-1}(x,t;k)-I=-\frac{\e^{-\frac{t\Phi(\xi,k_0)}{2}\hat{\sigma}_3}\delta^{\hat{\sigma}_3}_0M^X_1(\xi)}
{\sqrt{8t}(k-k_0)}+O\bigg(\frac{1}{t}\bigg),\quad t\rightarrow\infty,~k\in \partial D_\varepsilon(k_0).
\ee
The estimate \eqref{4.48} immediately follows from \eqref{4.57} and $|M_1^X|\leq C$. By Cauchy's formula and \eqref{4.57}, we derive \eqref{4.49}.
\end{proof}

\subsection{Derivation of the asymptotics formula and error bound}
We now begin to establish the explicit long-time asymptotics formula for the KE equation \eqref{1.1} on the half-line.

Define the approximate solution $M^{(app)}(x,t;k)$ by
\be\label{4.58}
M^{(app)}=\left\{\begin{aligned}
&M^{(k_0)},\quad k\in D_\varepsilon(k_0),\\
&I,\qquad ~~~{\text elsewhere}.
\end{aligned}
\right.
\ee
Let $\hat{M}(x,t;k)$ be
\be\label{4.58'}
\hat{M}=M^{(3)}(M^{(app)})^{-1},
\ee
 then $\hat{M}(x,t;k)$ satisfies the following RH problem
\be\label{4.59}
\hat{M}_+(x,t;k)=\hat{M}_-(x,t;k)\hat{J}(x,t,k),\quad k\in\hat{\Sigma},
\ee
where the jump contour $\hat{\Sigma}=\Sigma^{(3)}\cup\partial D_\varepsilon(k_0)$ is depicted in Fig. \ref{fig7}, and the jump matrix $\hat{J}(x,t,k)$ is given by
\be\label{4.60}
\hat{J}=\left\{
\begin{aligned}
&M^{(k_0)}_-J^{(3)}(M^{(k_0)}_+)^{-1},\quad k\in\hat{\Sigma}\cap D_\varepsilon(k_0),\\
&(M^{(k_0)})^{-1},\qquad\qquad\quad k\in\partial D_\varepsilon(k_0),\\
&J^{(3)},\qquad\qquad\qquad\quad~~ k\in\hat{\Sigma}\setminus \overline{D_\varepsilon(k_0)}.
\end{aligned}
\right.
\ee
Let $\hat{W}=\hat{J}-I.$ Then the following inequalities are valid.
\begin{figure}[htbp]
  \centering
  \includegraphics[width=4in]{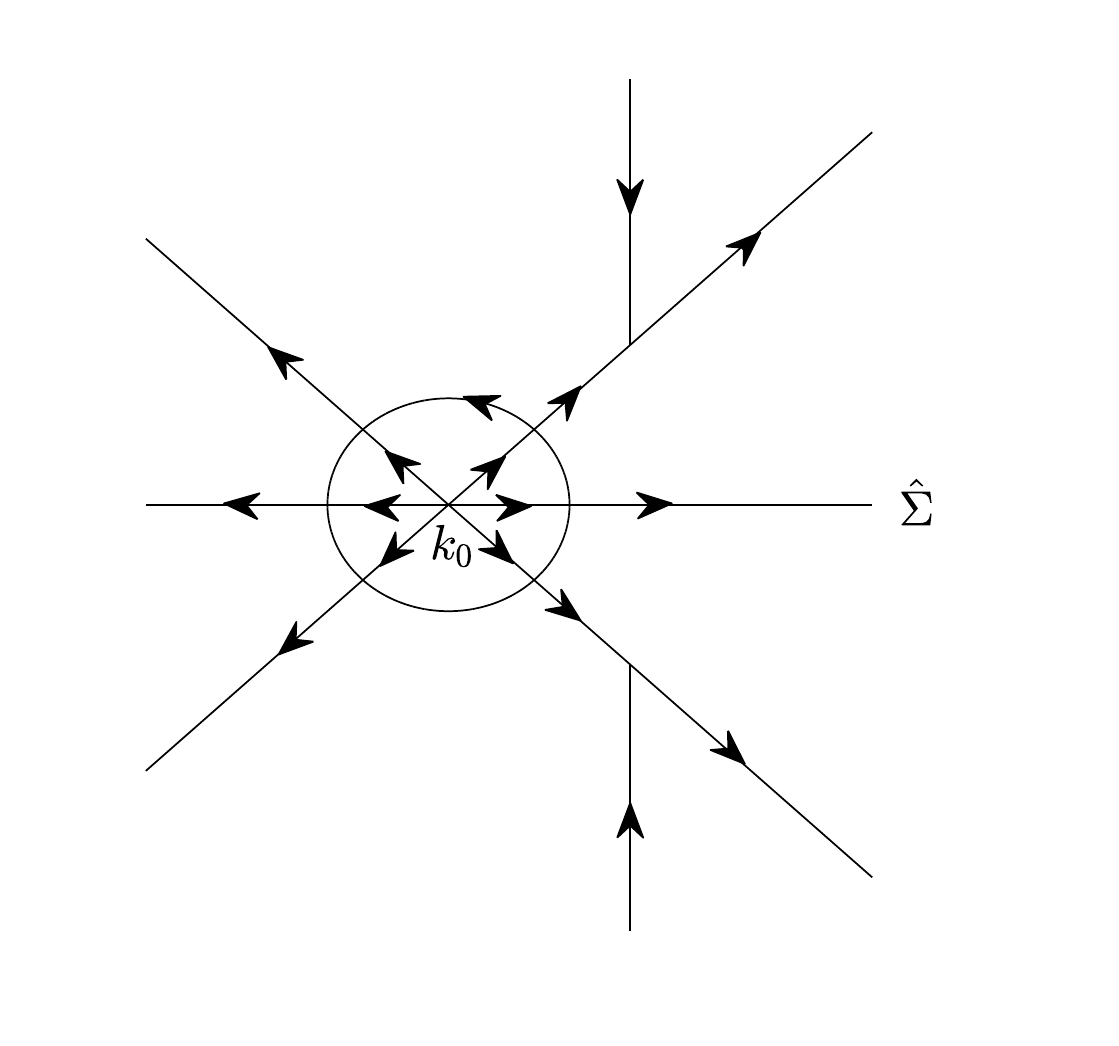}
  \caption{The contour $\hat{\Sigma}$.}\label{fig7}
\end{figure}
\begin{lemma}\label{lem4}
For $t>3$ and $\xi\in\mathcal{I}$, we have
\bea
\|\hat{W}\|_{L^1\cap L^2\cap L^\infty(\Sigma^{(3)}\setminus\mathcal{X})}&\leq Ct^{-3/2},\label{4.61}\\
\|\hat{W}\|_{L^1\cap L^2\cap L^\infty(\mathcal{X}\setminus\overline{D_\varepsilon(k_0)})}&\leq C\e^{-ct},\label{4.62}\\
\|\hat{W}\|_{L^1\cap L^2\cap L^\infty(\partial D_\varepsilon(k_0))}&\leq Ct^{-1/2},\label{4.63}
\eea
and
\bea\label{4.64}
\begin{aligned}
\|\hat{W}\|_{L^\infty(\mathcal{X}^\varepsilon)}&\leq Ct^{-1/2}\ln t,\\
\|\hat{W}\|_{L^1(\mathcal{X}^\varepsilon)}&\leq Ct^{-1}\ln t,\\
\|\hat{W}\|_{L^2(\mathcal{X}^\varepsilon)}&\leq Ct^{-3/4}\ln t.
\end{aligned}
\eea
\end{lemma}
\begin{proof}
Since the matrix $\hat{W}$ on $\Sigma^{(3)}\setminus\mathcal{X}$ only involves the small remainders $\Gamma_r$, $r_{1,r}$ and $r_{2,r}$, thus, by Lemmas \ref{lem1} and \ref{lem2}, the estimate \eqref{4.61} follows.

For $k\in D_2\cap(\mathcal{X}_1\setminus\overline{D_\varepsilon(k_0)})$, $\hat{W}$ only has a nonzero $-(r_{1,a}+\Gamma)\delta^{-2}\e^{t\Phi}$ in $(21)$ entry. Hence, for $t\geq1$, by \eqref{4.30}, we get
\bea
|\hat{W}_{21}|&=&|-(r_{1,a}+\Gamma)\delta^{-2}\e^{t\Phi}|\nn\\
&\leq& C|r_{1,a}+\Gamma|\e^{t\text{Re}\Phi}\nn\\
&\leq&C\e^{-3|k-k_0|^2t}\leq  C\e^{-3\varepsilon^2t}.\nn
\eea
In a similar way, the estimates on $\mathcal{X}_j\setminus\overline{D_\varepsilon(k_0)}$, $j=2,3,4$ hold.
This proves \eqref{4.62}.

The inequality \eqref{4.63} is a consequence of \eqref{4.48} and \eqref{4.60}.

For $k\in\mathcal{X}^\varepsilon$, we find $$\hat{W}=M^{(k_0)}_-(J^{(3)}-J^{(k_0)})(M^{(k_0)}_+)^{-1}.$$ Therefore, it follows from \eqref{4.46} and \eqref{4.47} that the estimate \eqref{4.64} holds.
\end{proof}

The results in Lemma \ref{lem4} imply that:
\be\label{4.65}
\begin{aligned}
\|\hat{W}\|_{L^\infty(\hat{\Sigma})}&\leq Ct^{-1/2}\ln t,\\
\|\hat{W}\|_{L^1\cap L^2(\hat{\Sigma})}&\leq Ct^{-1/2},
\end{aligned}
\quad t>3,~\xi\in\mathcal{I}.
\ee

 Let $\hat{C}$ denote the Cauchy operator associated with $\hat{\Sigma}$:
\berr
(\hat{C}f)(k)=\int_{\hat{\Sigma}}\frac{f(\zeta)}{\zeta-k}\frac{\dd \zeta}{2\pi\ii},\quad k\in\bfC\setminus\hat{\Sigma},~f\in L^2(\hat{\Sigma}).
\eerr
We denote the boundary values of $\hat{C}f$ from the left and right sides of $\hat{\Sigma}$ by $\hat{C}_+f$ and $\hat{C}_-f$, respectively. As is well known, the operators $\hat{C}_\pm$ are bounded from $L^2(\hat{\Sigma})$ to $L^2(\hat{\Sigma})$, and $\hat{C}_+-\hat{C}_-=I$, here $I$ denotes the identity operator.

Define the operator $\hat{C}_{\hat{W}}$: $L^2(\hat{\Sigma})+L^\infty(\hat{\Sigma})\rightarrow L^2(\hat{\Sigma})$ by $\hat{C}_{\hat{W}}f=\hat{C}_-(f\hat{W}),$ that is, $\hat{C}_{\hat{W}}$ is defined by $\hat{C}_{\hat{W}}(f)=\hat{C}_+(f\hat{W}_-)+\hat{C}_-(f\hat{W}_+)$ where we have chosen, for simplicity, $\hat{W}_+=\hat{W}$ and $\hat{W}_-=0$. Then, by \eqref{4.65}, we find
\be\label{4.66}
\|\hat{C}_{\hat{W}}\|_{B(L^2(\hat{\Sigma}))}\leq C\|\hat{W}\|_{L^\infty(\hat{\Sigma})}\leq Ct^{-1/2}\ln t,
\ee
where $B(L^2(\hat{\Sigma}))$ denotes the Banach space of bounded linear operators $L^2(\hat{\Sigma})\rightarrow L^2(\hat{\Sigma})$. Therefore, there exists a $T>0$ such that $I-\hat{C}_{\hat{W}}\in B(L^2(\hat{\Sigma}))$ is invertible for all $\xi\in\mathcal{I},$ $t>T$. Following this, we may define the $2\times2$ matrix-valued function $\hat{\mu}(x,t;k)$ whenever $t>T$ by
\be\label{4.67}
\hat{\mu}=I+\hat{C}_{\hat{W}}\hat{\mu}.
\ee
Then
\be\label{4.68}
\hat{M}(x,t;k)=I+\frac{1}{2\pi\ii}\int_{\hat{\Sigma}}\frac{(\hat{\mu}\hat{W})(x,t;\zeta)}{\zeta-k}\dd\zeta,\quad k\in\bfC\setminus\hat{\Sigma}
\ee
is the unique solution of the RH problem \eqref{4.59} for $t>T$.

Moreover, the function $\hat{\mu}(x,t;k)$ satisfies
\be\label{4.69}
\|\hat{\mu}(x,t;\cdot)-I\|_{L^2(\hat{\Sigma})}=O(t^{-1/2}),\quad t\rightarrow\infty,~\xi\in\mathcal{I}.
\ee
In fact, equation \eqref{4.67} is equivalent to $\hat{\mu}=I+(I-\hat{C}_{\hat{W}})^{-1}\hat{C}_{\hat{W}}I$. Using the Neumann series, we get
$$\|(I-\hat{C}_{\hat{W}})^{-1}\|_{B(L^2(\hat{\Sigma}))}\leq\frac{1}{1-\|\hat{C}_{\hat{W}}\|_{B(L^2(\hat{\Sigma}))}}$$
whenever $\|\hat{C}_{\hat{W}}\|_{B(L^2(\hat{\Sigma}))}<1$. thus, we find
\bea
\|\hat{\mu}(x,t;\cdot)-I\|_{L^2(\hat{\Sigma})}&=&\|(I-\hat{C}_{\hat{W}})^{-1}\hat{C}_{\hat{W}}I\|_{L^2(\hat{\Sigma})}\nn\\
&\leq&\|(I-\hat{C}_{\hat{W}})^{-1}\|_{B(L^2(\hat{\Sigma}))}\|\hat{C}_-(\hat{W})\|_{L^2(\hat{\Sigma})}\nn\\
&\leq&\frac{C\|\hat{W}\|_{L^2(\hat{\Sigma})}}{1-\|\hat{C}_{\hat{W}}\|_{B(L^2(\hat{\Sigma}))}}\leq C\|\hat{W}\|_{L^2(\hat{\Sigma})}\nn
\eea
for all $t$ large enough and all $\xi\in\mathcal{I}$. In view of \eqref{4.65}, this gives \eqref{4.69}.

It follows from \eqref{4.68} that
\be\label{4.70}
\lim_{k\rightarrow\infty}k(\hat{M}(x,t;k)-I)=-\frac{1}{2\pi\ii}\int_{\hat{\Sigma}}(\hat{\mu}\hat{W})(x,t;k)\dd k.
\ee
Denoting $\Sigma'=\Sigma^{(3)}\setminus\mathcal{X}^\varepsilon$, using \eqref{4.61} and \eqref{4.69}, we have
\bea
\int_{\Sigma'}(\hat{\mu}\hat{W})(x,t;k)\dd k&=&\int_{\Sigma'}\hat{W}(x,t;k)\dd k+\int_{\Sigma'}(\hat{\mu}(x,t;k)-I)\hat{W}(x,t;k)\dd k\nn\\
&\leq&\|\hat{W}\|_{L^1(\Sigma')}+\|\hat{\mu}-I\|_{L^2(\Sigma')}\|\hat{W}\|_{L^2(\Sigma')}\nn\\
&\leq&Ct^{-3/2},\quad t\rightarrow\infty.\nn
\eea
Similarly, by \eqref{4.64} and \eqref{4.69}, the contribution from $\mathcal{X}^\varepsilon$ to the right-hand side of \eqref{4.70} is $$O(\|\hat{W}\|_{L^1(\mathcal{X}^\varepsilon)}+\|\hat{\mu}-I\|_{L^2(\mathcal{X}^\varepsilon)}
\|\hat{W}\|_{L^2(\mathcal{X}^\varepsilon)})=O(t^{-1}\ln t),\quad t\rightarrow\infty.$$
Finally, by \eqref{4.49}, \eqref{4.63} and \eqref{4.69}, we can get
\bea
&&-\frac{1}{2\pi\ii}\int_{\partial D_\varepsilon(k_0)}(\hat{\mu}\hat{W})(x,t;k)\dd k\nn\\
&=&-\frac{1}{2\pi\ii}\int_{\partial D_\varepsilon(k_0)}\hat{W}(x,t;k)\dd k-\frac{1}{2\pi\ii}\int_{\partial D_\varepsilon(k_0)}(\hat{\mu}(x,t;k)-I)\hat{W}(x,t;k)\dd k\nn\\
&=&-\frac{1}{2\pi\ii}\int_{\partial D_\varepsilon(k_0)}\bigg((M^{(k_0)})^{-1}(x,t;k)-I\bigg)\dd k+O(\|\hat{\mu}-I\|_{L^2(\partial D_\varepsilon(k_0))}
\|\hat{W}\|_{L^2(\partial D_\varepsilon(k_0))})\nn\\
&=&\frac{\e^{-\frac{t\Phi(\xi,k_0)}{2}\hat{\sigma}_3}\delta^{\hat{\sigma}_3}_0M^X_1(\xi)}
{\sqrt{8t}}+O(t^{-1}),\quad t\rightarrow\infty.\nn
\eea
Thus, we obtain the following important relation
\be\label{4.71}
\lim_{k\rightarrow\infty}k(\hat{M}(x,t;k)-I)=\frac{\e^{-\frac{t\Phi(\xi,k_0)}{2}\hat{\sigma}_3}\delta^{\hat{\sigma}_3}_0M^X_1(\xi)}
{\sqrt{8t}}+O(t^{-1}\ln t),\quad t\rightarrow\infty.
\ee

Taking into account that \eqref{3.9}, \eqref{4.6}, \eqref{4.8}, \eqref{4.31} and \eqref{4.58'}, for sufficient large $k\in\bfC\setminus\hat{\Sigma}$, we get
\bea\label{4.72}
m(x,t)&=&\lim_{k\rightarrow\infty}(kM(x,t;k))_{12}\nn\\
&=&\lim_{k\rightarrow\infty}k(\hat{M}(x,t;k)-I)_{12}\nn\\
&=&\frac{\big(\e^{-\frac{t\Phi(\xi,k_0)}{2}\hat{\sigma}_3}\delta^{\hat{\sigma}_3}_0M^X_1(\xi)\big)_{12}}
{\sqrt{8t}}+O\bigg(\frac{\ln t}{t}\bigg)\\
&=&\frac{-\ii\beta^X\delta_0^2\e^{-t\Phi(\xi,k_0)}}{\sqrt{8t}}+O\bigg(\frac{\ln t}{t}\bigg)\nn\\
&=&\frac{-\ii\beta^X(8t)^{-\ii\nu}\e^{2\chi(k_0)}\e^{4\ii tk_0^2}}{\sqrt{8t}}+O\bigg(\frac{\ln t}{t}\bigg),\quad t\rightarrow\infty.\nn
\eea
From \eqref{3.9}, we have
\bea
u(x,t)=2\ii m(x,t)\e^{2\ii\int_{(0,0)}^{(x,t)}\Delta}.\nn
\eea
Recalling the definition of $\Delta(x,t)$ in \eqref{2.7},
$$\Delta(x,t)=-\beta|u|^2\dd x+\bigg(4\beta^2|u|^4-\ii\beta(u_x\bar{u}-u\bar{u}_x)\bigg)\dd t.$$
In order to compute the integral $\int_{(0,0)}^{(x,t)}\Delta$, we choose an integration contour consisting of the vertical segment from $(0,0)$ to $(0,t)$ followed by the horizontal segment from $(0,t)$ to $(x,t)$, and using the equality $|u(x,t)|^2=4|m(x,t)|^2$, we get
\be\label{4.73}
\int_{(0,0)}^{(x,t)}\Delta=\int_0^t\bigg(4\beta^2|g_0(t')|^4-\ii\beta(g_1\bar{g}_0-g_0\bar{g}_1)(t')\bigg)\dd t'-4\beta\int_0^x|m(x',t)|^2\dd x'.
\ee
It follows from \eqref{4.72} and the definition $\beta^X(q)$ in \eqref{4.43}, we have
\bea
\int_0^x|m(x',t)|^2\dd x'&=&\int_0^x\bigg|\sqrt{\frac{\nu(x'/t)}{8t}}+O\bigg(\frac{\ln t}{t}\bigg)\bigg|^2\nn\\
&=&-\frac{1}{16\pi t}\int_0^x\ln\big(1-|r(-x'/(4t))|^2\big)\dd x'+O\bigg(\frac{\ln t}{t^{3/2}}\bigg)\\
&=&-\frac{1}{4\pi}\int_0^{|k_0|}\ln\big(1-|r(-s)|^2\big)\dd s+O\bigg(\frac{\ln t}{t^{3/2}}\bigg).\nn
\eea
Collecting the above computations, we obtain our main results stated as the following theorem.
\begin{theorem}\label{the4.2}
Let $u_0(x),g_0(t),g_1(t)$ lie in the Schwartz space $S([0,\infty))$. Suppose the assumption \ref{ass1} be valid. Then, for any positive constant $N$, as $t\rightarrow\infty$, the solution $u(x,t)$ of the IBV problem for KE equation \eqref{1.1} on the half-line satisfies the following asymptotic formula
\be
u(x,t)=\frac{u_a(x,t)}{\sqrt{t}}+O\bigg(\frac{\ln t}{t}\bigg),\quad t\rightarrow\infty,~0\leq x\leq Nt,
\ee
where the error term is uniform with respect to $x$ in the given range, and the leading-order
coefficient $u_a(x,t)$ is defined by
\be
u_a(x,t)=\sqrt{\frac{\nu(\xi)}{2}}\e^{\ii\alpha(\xi,t)},
\ee
with
$$\xi=\frac{x}{t},~\nu(\xi)=-\frac{1}{2\pi}\ln(1-|r(k_0)|^2),~k_0=-\frac{\xi}{4},$$ and
\bea
\alpha(\xi,t)&=&-\frac{3\pi}{4}-\arg r(k_0)+\arg\Gamma(\ii\nu(\xi))-\nu(\xi)\ln(8t)+4k_0^2t\nn\\
&&+\frac{1}{\pi}\int_{-\infty}^{k_0}\ln(k_0-s)\dd\ln(1-|r(s)|^2)+\frac{2\beta}{\pi}\int_0^{|k_0|}\ln(1-|r(-s)|^2)\dd s\nn\\
&&+2\int_0^t\bigg(4\beta^2|g_0(t')|^4-\ii\beta(g_1\bar{g}_0-g_0\bar{g}_1)(t')\bigg)\dd t'.\nn
\eea
\end{theorem}

\medskip
\small{

}
\end{document}